\documentclass[11pt, reqno]{amsart}
\usepackage[utf8]{inputenc}
\usepackage{color}
\usepackage{amsthm}
\usepackage{enumerate}
\usepackage{dcpic,pictex} 
\usepackage[all]{xy}

\usepackage{version}
\usepackage{amsmath,amsfonts,amssymb}
\usepackage{mathtools}
\usepackage{mathrsfs}
\usepackage{hyperref}
\usepackage{nicefrac}

\newtheorem{teo}{Theorem}[section]
\newtheorem{lemma}[teo]{Lemma}
\newtheorem{prop}[teo]{Proposition}

\theoremstyle{definition}
\newtheorem{defi}[teo]{Definition}
\newtheorem{definition}[teo]{Definition}

\newtheorem{remark}[teo]{Remark}

\newcommand{\R}{\mathbb R}

\newcommand{\C}{\mathbb C}

\newcommand{\B}{\mathbb B}

\newcommand{\D}{\mathbb D}

\newcommand{\N}{\mathbb N}

\newcommand{\HH}{\mathbb H}

\def\v{\varphi}

\theoremstyle{remark} 

\numberwithin{equation}{section}

\begin{document}

\title[Models via rescaling]{Canonical models on strongly convex domains\\ via the squeezing function}

\author[A. Altavilla]{Amedeo Altavilla$^\dag$}\address{A. Altavilla: Dipartimento di Matematica, Universit\`a degli Studi di Bari Aldo Moro, Via E. Orabona 4, 70125 Bari, Italy}\email{amedeo.altavilla@uniba.it}
\author[L. Arosio]{Leandro Arosio$^\ddag$}
\address{L. Arosio: Dipartimento Di Matematica, Universit\`a di Roma "Tor Vergata", Via Della Ricerca Scientifica 1, 00133, Roma, Italy} 
\email{arosio@mat.uniroma2.it}
\author[L. Guerini]{Lorenzo Guerini}
\address{L. Guerini: Korteweg de Vries Institute for Mathematics, University of Amsterdam, Science Park 107, 1090GE Amsterdam, the Netherlands}
\email{lorenzo.guerini92@gmail.com}

\thanks{$^{\dag}$
GNSAGA of INdAM; supported by the SIR grant {\sl ``NEWHOLITE - New methods in holomorphic iteration''} n. RBSI14CFME and SIR grant {\sl AnHyC - Analytic aspects in complex and hypercomplex geometry} n. RBSI14DYEB}
\thanks{$^{\ddag}$  Supported by the SIR grant {\sl``NEWHOLITE - New methods in holomorphic iteration''} no. RBSI14CFME. Partially supported by the MIUR Excellence Department Project awarded to the Department of Mathematics, University of Rome Tor Vergata, CUP E83C18000100006.}
\date{\today}

\begin{abstract} 
We prove that if a holomorphic self-map $f\colon \Omega\to \Omega$ of a bounded  strongly convex domain  $\Omega\subset \C^q$ with smooth boundary is  hyperbolic then it admits a natural semi-conjugacy with a hyperbolic automorphism of a possibly lower dimensional ball $\B^k$. We also obtain the dual result for a holomorphic self-map $f\colon \Omega\to \Omega$ with a boundary repelling fixed point.
Both results are obtained by rescaling the dynamics of $f$ via the squeezing function.
\end{abstract}
\maketitle
{\small\tableofcontents}
\section{Introduction}
%
When  studying the dynamics of a holomorphic self-map $f$ of the unit ball $\B^q\subset \C^q$,
an important role is played  by  fixed points at the boundary, where the map is not necessarily continuous.
A point $\zeta\in \partial \B^q$ is a {\sl boundary regular fixed point}  if
\begin{enumerate}
\item  for every sequence $(z_n)$ converging to $\zeta$ inside a Koranyi region
$$K(\zeta,M):=\left\{z\in \B ^q\,\Big|\,\frac{|1-\langle z,\zeta\rangle|}{1-\Vert z\Vert}<M\right\},$$
where $M>1$, we have that $f(z_n)$ converges to $\zeta$, and
\item   the {\sl dilation} $\lambda_\zeta$ defined as
\begin{equation}\label{dilationball}
\lambda_\zeta:=\liminf_{z\to \zeta}\frac{1-\|f(z)\|}{1-\|z\|},
\end{equation}
is finite. If $\lambda_\zeta>1$ the point $\zeta$ is called a {\sl boundary repelling fixed point}.
\end{enumerate}

%
%
%

The classical Denjoy--Wolff Theorem illustrates the relevance of this notion.
\begin{teo}
 Let $f\colon \B^q\to \B^q$ be a holomorphic self-map without interior fixed points. Then 
there exists a boundary regular fixed point $\xi\in \partial \B^q$ with dilation  
$0<\lambda_\xi\leq 1$, called the {\sl Denjoy--Wolff point}, such that  the sequence of iterates $(f^n)$ converges to $\xi$.
\end{teo}
As a consequence, the family of holomorphic self-maps of $\B^q$ is partitioned in three classes:  $f$ is {\sl elliptic} if it admits a fixed point $z\in \B^q$, and if $f$ is not elliptic, it is {\sl parabolic} if the dilation $\lambda_\xi$ at its Denjoy--Wolff point  is 1 and it is {\sl hyperbolic} if $\lambda_\xi<1$.

The automorphisms of $\B^q$  have explicit normal forms, which show that they have a simple dynamical behaviour. For example, a hyperbolic automorphism has only two boundary regular fixed points, one is the Denjoy--Wolff $\xi$, and the other is a boundary repelling fixed point $\zeta$ with dilation $\lambda_\zeta=1/\lambda_\xi.$ 
The normal form of hyperbolic automorphisms is easily described. Recall that the {\sl Siegel half-space} $\HH^q:=\{(z_1,z')\in \C\times{\C^{q-1}}\colon{\rm Im}\,z_1>\|z'\|^2\}$ is biholomorphic to the ball $\B^q$.
Given any hyperbolic automorphism $\tau$ of the ball there exists a biholomorphism $\Psi\colon \B^q\to \HH^q$ sending the Denjoy--Wolff point $\xi$ to $\infty$, such that
$$\Psi\circ \tau\circ \Psi^{-1}(z_1,z')=\left(\frac{1}{\lambda_\xi} z_1,\frac{e^{it_1}}{\sqrt \lambda_\xi}z'_1,\dots, \frac{e^{it_{q-1}}}{\sqrt \lambda_\xi}z'_{q-1} \right),$$
with $t_j\in \R$ (a similar normal form can be obtained sending the repelling point to $\infty$). 

In order to understand the forward or backward dynamics of a holomorphic self-map $f$ it is natural to search for a semi-conjugacy between $f$ and an automorphism of the ball. In this direction, the following results were recently proved in \cite{ Ar, Ar2, ArBr, ArGu} using the theory of canonical models (the cases $q=1$  are the classical  results of Valiron \cite{Valiron} and Poggi-Corradini \cite{PC1}).
\begin{teo}[Forward iteration]\label{teopallaavanti}
Let $f\colon \B^q\to \B^q$ be a hyperbolic holomorphic self-map with  Denjoy--Wolff point $\xi$. Then there exists an integer $1\leq k\leq q$, a holomorphic map $h\colon \B^q\to \HH^k$ and a hyperbolic automorphism $\tau$ of $\HH^k$ of the form
$$\tau(z_1,z')=\left(\frac{1}{\lambda_\xi} z_1,\frac{e^{it_1}}{\sqrt \lambda_\xi}z'_1,\dots, \frac{e^{it_{k-1}}}{\sqrt \lambda_\xi}z'_{k-1} \right)$$
such that $$h\circ f= \tau\circ h.$$  
\end{teo}
\begin{teo}[Backward iteration]\label{teopallaindietro}
Let $f\colon \B^q\to \B^q$ be a  holomorphic self-map and let $\zeta$ be a boundary repelling fixed point.
Then there exists an integer $1\leq k\leq q$, a holomorphic map $h\colon\HH^k\to \B^q$ and a hyperbolic automorphism $\tau$ of $\HH^k$ of the form
$$\tau(z_1,z')=\left(\frac{1}{\lambda_\zeta} z_1,\frac{e^{it_1}}{\sqrt \lambda_\zeta}z'_1,\dots, \frac{e^{it_{k-1}}}{\sqrt \lambda_\zeta}z'_{k-1} \right)$$
such that 
$$f\circ h= h\circ \tau.$$
\end{teo}

In both theorems the function $h$ intertwines the map $f$ with a hyperbolic holomorphic automorphism of $\HH^k$. In 
Theorem \ref{teopallaavanti} one obtains an automorphism with Denjoy--Wolff point at $\infty$ and dilation $\lambda_\infty=\lambda_\xi$. In Theorem \ref{teopallaindietro} one obtains an automorphism with a repelling boundary fixed point at $\infty$ with dilation $\lambda_\infty=\lambda_\zeta$. The semi-conjugacy provided by $h$ satisfies a universal property and thus is  unique up to biholomorphisms. Such a semi-conjugacy is  called a {\sl canonical model} for $f$, see Sections \ref{canonicalforward}, \ref{canonicalbackward} for  definitions.

In this paper we are interested in extending these results to the case where $f\colon \Omega\to \Omega$ is a holomorphic self-map of a strongly convex domain  $\Omega\subset \subset\C^q$ whose boundary is  $C^3$. 
For such map $f$, the concepts of  boundary regular fixed point and dilation can be defined intrinsically in terms of the Kobayashi distance $k_\Omega$, see Section \ref{background} for definitions. 
In \cite{Ab2} Abate proved that  the Denjoy--Wolff theorem still holds in this setting. Hence we can partition the family of holomorphic self-maps of $\Omega$ in elliptic, parabolic and hyperbolic maps  exactly as in the case of the ball.

Let thus $f\colon \Omega\to \Omega$ be a holomorphic self-map of a strongly convex domain, which is either hyperbolic or which admits a boundary repelling fixed point $\zeta$. When trying to generalize Theorems  \ref{teopallaavanti} and \ref{teopallaindietro}, the first  obstacle that one encounters is that the proofs of these theorems rely heavily on the fact that the automorphism group of the ball is transitive, whereas  by Wong--Rosay's theorem any strongly convex domain which is not biholomorphic to the ball cannot have a transitive group of automorphisms. Moreover, it is  natural to search  for a semi-conjugacy of $f$ with a  hyperbolic automorphism of a (possibly lower-dimensional) strongly convex domain   $\Lambda\subset \subset \C^k$, but it follows again by   
Wong--Rosay's theorem that if a  strongly convex domain   $\Lambda$ admits a non-elliptic automorphism, then $\Lambda$ is biholomorphic to the ball $\B^k$.

Indeed, we prove that $f$  admits a natural semi-conjugacy with a hyperbolic automorphism of a ball $\B^k$, where $1\leq k\leq q$.
To cope with the lack of transitivity, we use the fact that the squeezing function $S_\Omega$ of $\Omega$ converges to $1$ at the boundary $\partial \Omega$, which roughly speaking means that the geometry of $\Omega$ resembles more and more the geometry of the ball as we approach the boundary. 
Hence we can rescale the dynamics of $f$  obtaining in the limit the desired intertwining mappings with automorphisms of a (possibly lower-dimensional) ball. 

Our main results are the following.
\begin{teo}[Forward iteration]\label{teoconvessoavanti}
Let $\Omega\subset\subset \C^q$ be a strongly convex domain with $C^3$ boundary.
Let $f\colon \Omega\to \Omega$ be a hyperbolic holomorphic self-map with  Denjoy--Wolff point $\xi$. Then there exists an integer $1\leq k\leq q$, a holomorphic map $h\colon \Omega\to \HH^k$ and a hyperbolic automorphism $\tau$ of $\HH^k$ of the form
$$\tau(z_1,z')=\left(\frac{1}{\lambda_\xi} z_1,\frac{e^{it_1}}{\sqrt \lambda_\xi}z'_1,\dots, \frac{e^{it_{k-1}}}{\sqrt \lambda_\xi}z'_{k-1} \right)$$
such that $$h\circ f= \tau\circ h.$$  
\end{teo}
\begin{teo}[Backward iteration]\label{teoconvessoindietro}
Let $\Omega\subset\subset \C^q$ be a strongly convex domain with $C^4$ boundary.
Let $f\colon \Omega\to \Omega$ be a  holomorphic self-map and let $\zeta$ be a boundary repelling fixed point.
Then there exists an integer $1\leq k\leq q$, a holomorphic map $h\colon \HH^k\to \Omega$ and a hyperbolic automorphism $\tau$ of $\HH^k$  of the form
$$\tau(z_1,z')=\left(\frac{1}{\lambda_\zeta} z_1,\frac{e^{it_1}}{\sqrt \lambda_\zeta}z'_1,\dots, \frac{e^{it_{k-1}}}{\sqrt \lambda_\zeta}z'_{k-1} \right)$$
such that 
$$f\circ h= h\circ \tau.$$
\end{teo}
In both cases the semi-conjugacy provided by $h$ is again a canonical model for $f$ and as such it is  unique up to biholomorphisms. Notice also that in the forward iteration case we obtain a similar result also for parabolic nonzero-step maps, see Theorem \ref{parabolicresult}.

The strong convexity of the domain $\Omega$, which implies the fact that the squeezing function converges to 1 at the boundary $\partial \Omega$, is  essential  in our proofs.
Indeed, a key ingredient of the proof of Theorem \ref{teoconvessoavanti} is that every holomorphic self-map of $\Omega$ with an orbit converging to the boundary admits a canonical model biholomorphic to a (possibly lower-dimensional) ball. This is not true on a weakly convex domain.
 As an example,
consider  the {\rm egg domain} $$\Omega:=\{(z_1,z_2)\in \C^2\colon |z_1|^2+|z_2|^4<1\},$$ 
which is strongly convex at every point of $\partial \Omega$ except for those with $\{z_2=0\}$ (where the squeezing function $S_\Omega$ does not converge to $1$).
Consider the automorphism $f\colon \Omega\to \Omega$ defined changing holomorphic coordinates to the unbounded realization $\{(z_1,z_2)\colon {\sf Im}\,z_1 >|z_2|^4\}$ of $\Omega$ and considering  $(z_1,z_2)\mapsto (\frac{1}{\lambda} z_1,\frac{1}{\sqrt[4]{\lambda}}z_2)$, with $0<\lambda<1$.
Every forward orbit of the automorphim $f\colon \Omega\to \Omega$ converges to the point $(1,0)$.
 A canonical model for the automorphism $f$ is simply given by the identity map ${\sf id}\colon \Omega\to \Omega$ intertwining $f$ with itself. But the canonical model is unique up to biholomorphisms, and $\Omega$ is not biholomorphic to the ball.
Similar considerations hold in the backward iteration case.

When dealing with the backward iteration case, an additional difficulty arises in the proof of Theorem \ref{teoconvessoindietro}. Indeed, in order to apply Theorem \ref{backastratto} one needs  to  construct  a backward orbit $(z_n)$ converging to $\zeta$ and satisfying
$$\lim_{n\to\infty} k_{\Omega}(z_n,z_{n+1})=\log\lambda_\zeta.$$
In  the case of the ball $\B^q$, such orbit is obtained in \cite{ArGu} using a fixed horosphere centered in $\zeta$ to define the stopping time of an iterative process. Transitivity of the automorphism group of $\B^q$ guarantees the convergence of this process.  In the case of a strongly convex domain $\Omega$ we need first to find a change of coordinates  in ${\rm Aut}(\C^q)$ so that in the new coordinates the  domain $\Omega$ contains a $\B^q$-horosphere centered in $e_1$, and is locally contained in $\B^q $ near $e_1$.
This is done  assuming that  $\partial\Omega$ is $C^4$-smooth, composing Fefferman's change of coordinates with a parabolic ``push", and using Anders\'en-Lempert jet interpolation to obtain an automorphism of $\C^q$. As a consequence we obtain that $k_{\B^q}$ and $k_\Omega$ are very close on small $\B^q$-horospheres centered in $e_1$.
We then define an iterative process using smaller and smaller   $\B^q$-horospheres   as stopping times, and we prove the convergence of the process by rescaling it with automorphisms of the ball.

\medskip {\bf Acknowledgements.} We want to thank Luka Boc Thaler for useful discussions and Andrew Zimmer for suggesting the proof of Proposition \ref{filippo}.

\section{Background}\label{background}
\subsection{Real geodesics}
\begin{definition}
Let $(X,d)$ be a metric space.  A {\sl real geodesic} is a map $\gamma$ from an interval $I\subset \R$ to $X$ which is an isometry with respect to the euclidean distance on $I$ and the distance on $X$, that is for all $s,t\in I$, $$d(\gamma(s),\gamma(t))=|t-s|.$$
If the interval is closed and bounded (resp. $[0,+\infty)$,  $(-\infty,+\infty)$) we call $\gamma$ a {\sl geodesic segment}  (resp. {\sl geodesic ray}, {\sl geodesic line}).
\end{definition}
\subsection{The ball} 
Horospheres and Koranyi regions play a central role in the study of strongly convex domains. They are generalizations of corresponding notions appearing in the setting of the unit ball $\B^q$, where they are defined in euclidean terms. Here we recall their definitions.

The {\sl horosphere} of  {\sl center} $\zeta\in\partial \B^q$ and {\sl radius} $R>0$ is  defined as
$$E(\zeta,R):=\left\{z\in \B ^q\colon \frac{|1-\langle z,\zeta\rangle|^2}{1-\Vert z\Vert^2}<R\right\}.$$
The {\sl Koranyi region} of {\sl center} $\zeta\in\partial \B^q$ and {\sl amplitude} $M>1$ is defined as
$$K(\zeta,M):=\left\{z\in \B ^q\colon\frac{|1-\langle z,\zeta\rangle|}{1-\Vert z\Vert}<M\right\}.$$
When working with horospheres, it is sometime convenient to consider their expression in the {\sl Siegel half-space} $\HH^q:=\{(z_1,z')\in\mathbb C\times \C^{q-1}\colon \textrm{Im}\, z_1> \Vert z'\Vert^2\}$, which is biholomorphic to $\B^q$ under the {\sl Cayley transform} $\mathcal C\colon \B^q\to \HH^q$ 
\begin{equation}
\label{cayley}
\mathcal C(z_1,z'):=\left(i\frac{1+z_1}{1-z_1},\frac{z'}{1-z_1}\right).
\end{equation}
Notice that also the bihomolomorphism $(z_1,z')\mapsto \left(i\frac{1+z_1}{1-z_1},i\frac{z'}{1-z_1}\right)$ from $\B^q$ to $\HH^q$  is commonly referred to as the Cayley transform.

With this change of holomorphic coordinates, the point $e_1=(1,0,\dots,0)$  is sent to $\infty$ and the horosphere $E(e_1, R)$ becomes 
$$
E(\infty,R):=\left\{(z_1,z')\in \HH ^q\colon \textrm{Im}\, z_1>\Vert z'\Vert^2+\frac{1}{R}\right\}.
$$

The automorphism group $\rm{Aut}(\B^q)$ is transitive, and this property characterizes the ball among strongly pseudoconvex domains.
\begin{teo}[Wong--Rosay \cite{Wo}]
Let $\Omega\subset \C^q$ be a domain. Suppose that there exist $x_0\in\Omega$ and a sequence  $(\varphi_n)$ in $\rm{Aut}(\Omega)$ such that $\varphi_n(x_0)\to \zeta\in\partial \Omega$, and that $\partial \Omega$ is $C^2$ and strongly pseudoconvex near $\zeta$. Then $\Omega$ is biholomorphic to $\B^q$.
\end{teo}

\subsection{Strongly convex domains}

 We start by recalling the definition of \textit{strong convexity}.

\begin{definition}
A bounded convex domain $\Omega\subset\C^q$ with $C^2$ boundary  is {\sl strongly convex} at $\zeta\in \partial \Omega$ if 
for some (and hence for any) defining function $\rho$ for $\Omega$ at $\zeta$, the Hessian $H_\zeta\rho$ is positive definite on the tangent space $T_\zeta\partial\Omega$. The domain $\Omega$ is {\sl strongly convex} if it is strongly convex at every point $\zeta\in \partial\Omega$.
\end{definition}
\begin{remark}
It is well known that a strongly convex domain is also strongly pseudoconvex.
\end{remark}
The analysis of strongly convex domains relies extensively on Lempert's theory of complex geodesics \cite{L}.
\begin{definition} A {\sl complex geodesic} in a Kobayashi hyperbolic manifold $X$ is a holomorphic map $\varphi\colon \D\rightarrow X$ which is an isometry with respect to the Kobayashi distance of the disc $\mathbb{D}\subset \C$ and the Kobayashi distance of $X$.
\end{definition}

\begin{teo}[See e.g. \cite{Ab}]\label{complexgeo}
Let $\Omega\subset \C^q$ be a bounded strongly convex domain with $C^3$ boundary.
\begin{enumerate}
\item For every pair of distinct points $z,w\in \Omega$ , let $r:={\rm tanh}(\frac{1}{2}k_\Omega(z,w))$. Then there exists a unique complex geodesic $\varphi$ such that $\varphi(0)=z$ and $\varphi(r)=w$.
\item  Every complex geodesic $\varphi\colon\D\to \Omega$ extends to a $C^1$ map on $\overline \D$, and the extension is injective.
\item For every $z\in \Omega$ and $\zeta\in\partial \Omega$, there exists a unique complex geodesic with $\varphi(0)=z$ and $\varphi(1)=\zeta$.
\end{enumerate}
\end{teo}

Complex geodesics are isometries between $(\D,k_\D)$ and $(\Omega,k_\Omega)$, and therefore map real geodesic of $\D$ to real geodesic in $\Omega$. On the other hand every real geodesic of $\Omega$ is contained in some complex geodesic \cite[Lemma 3.3]{GS}, and its pullback is a real geodesic in $\D$.
 Hence every real geodesic in $\Omega$ is $C^\infty$, and for all $x\neq y\in \Omega$ the geodesic segment joining $x$ to $y$ is unique up to isometries of the interval $I$. 
Moreover for every $p\in\Omega$ and $\zeta\in\partial\Omega$ there exists a unique geodesic ray connecting the two points.

 Let $\zeta\in\partial \Omega$ and choose a {\sl pole} $p\in \Omega$. It is proved in \cite[Theorem 2.6.47]{Ab}  that the limit 
\begin{equation}\label{limiteesiste}
\lim_{w\to\zeta}[k_\Omega(z,w)-k_\Omega(p,w)]
\end{equation} exists. We denote by  
$h_{\zeta,p}\colon \Omega\rightarrow \mathbb R_{>0}$ (sometimes by $h^\Omega_{\zeta,p}$) the continuous function defined as
$$
 h_{\zeta,p}(z):=\textrm{exp}\left(\lim_{w\to\zeta}[k_\Omega(z,w)-k_\Omega(p,w)]\right).
$$
If instead of $p$ we choose a different pole $p'\in \Omega$ the function changes by a multiplicative constant:
\begin{equation}
\label{differentpole}
h_{\zeta,p'}(z)=h_{\zeta,p}(z)h_{\zeta,p'}(p).
\end{equation}
The concepts of horosphere, Koranyi region, boundary regular fixed points and dilation can be carried over to the case of strongly convex domains with $C^3$ boundary, giving intrinsic definitions in terms of the Kobayashi distance.

\begin{definition}
The {\sl horosphere} of {\sl center} $\zeta\in\partial \Omega$, {\sl pole} $p\in \Omega$ and {\sl radius} $R>0$ is the set
\[
E_\Omega(p,\zeta,R):=\left\{z\in \Omega\,|\,h_{\zeta,p}(z)<R\right\}.
\]
The {\sl Koranyi region}  of {\sl center} $\zeta\in\partial \Omega$, {\sl pole} $p\in \Omega$ and {\sl amplitude} $M>1$ is the set
\[
K_\Omega(p,\zeta,M):=\left\{z\in \Omega\,|\,\log h_{\zeta,p}(z)+k_\Omega(p,z)<2\log M\right\}.
\]
\end{definition}

\begin{remark}
When $\Omega=\B^q$, horospheres and Koranyi regions with pole $p=0$ and center $\zeta\in\partial \B^q$ coincide with the regions $E(\zeta,R)$ and $K(\zeta,M)$ defined previously (see \cite[Propositions 2.2.20 and 2.7.3]{Ab}).
\end{remark}

\begin{definition}\label{defdil}
Let $\Omega\subset \C^q $ be a strongly convex domain with $C^3$ boundary.
A holomorphic map $f\colon \Omega \to \C^m$ has $K$-limit $\sigma$ at $\zeta\in\partial \Omega$ if
for every sequence $(z_n)$ converging to $\zeta$ inside a Koranyi region we have that $f(z_n)$ converges to $\sigma$.
If $f\colon \Omega\to \Omega$ is a holomorphic self-map, a point $\zeta\in \partial\Omega $ is a {\sl boundary fixed point} if 
$$K\hbox{-}\lim_{z\to \zeta}f(z)=\zeta.$$
Given $\zeta\in \partial \Omega$, the  {\sl dilation} of $f$ at $\zeta$ with {\sl pole} $p\in \Omega$ is the number  $\lambda_{\zeta,p}\in\mathbb R_{>0} $ defined as
\[
\log \lambda_{\zeta,p}=\liminf_{z\to\zeta}[k_\Omega(p,z)-k_\Omega(p,f(z))].
\]
\end{definition}
\begin{remark}
By \cite[Lemma 1.3]{AbRa} the dilation coefficient $\lambda_{\zeta,p}$ at a boundary fixed point does not depend on $p\in \Omega$, thus we can write $\lambda_\zeta=\lambda_{\zeta,p}.$ The proof of this fact is  based on the existence of complex geodesics and of the limit (\ref{limiteesiste}).
\end{remark}
\begin{definition}
A boundary fixed point $\zeta\in \partial \Omega$ for $f\colon \Omega \to \Omega$ is {\sl regular} if its dilation $\lambda_\zeta$ is finite.
\end{definition}
\begin{remark}
If $\Omega=\B^q$, then a straightforward calculation shows that
$$\liminf_{z\to\zeta}[k_\Omega(0,z)-k_\Omega(0,f(z))]=\liminf_{z\to\zeta}\log\frac{1-\Vert f(z)\Vert}{1-\Vert z\Vert},$$
in agreement with \eqref{dilationball}.
\end{remark}

We will need the following version of Julia's Lemma (see e.g. \cite[Theorem 2.4.16, Proposition 2.7.15]{Ab})
\begin{prop}
\label{julialem}
Let $\Omega\subset \C^q$ be a bounded strongly convex domain with $C^3$ boundary, and let  $f\colon\Omega\to \Omega$ be a holomorphic self-map. Let $\zeta\in\partial \Omega$ be a boundary regular fixed point, and let $p\in \Omega$. Then
\begin{align*}
f(E_\Omega(p,\zeta,R))\subset E_\Omega(p,\zeta,\lambda_{\zeta}R),\quad \forall R>0.
\end{align*}
\end{prop}

Complex geodesics are also useful in order to compute dilation coefficients. Recall for example the following result \cite[Lemma 3.1]{AbRa}.
\begin{lemma}
\label{AbateRaissy}
Let $\Omega\subset \C^q$ be a bounded strongly convex domain with $C^3$ boundary, and let $f\colon \Omega\rightarrow \Omega$ be a holomorphic self-map. Let $\zeta\in\partial \Omega$ be a boundary regular fixed point of $f$, and let $\varphi\colon \D\rightarrow \Omega$ be  a complex geodesic with $\varphi(1)=\zeta$. Then
$$
\lim_{t\to 1, t\in \R\cap \D}k_\Omega(\varphi(t),f(\varphi(t)))=|\log \lambda_\zeta|.
$$

\end{lemma}

The Denjoy--Wolff theorem also carries over to this setting (see e.g. \cite[Theorem 0.6]{Ab2}).
\begin{teo} Let $\Omega\subset \C^q$ be a bounded strongly convex $C^3$ domain. Let $f\colon \Omega\to \Omega$ be a holomorphic self-map without interior fixed points.
Then 
there exists a boundary regular fixed point $\xi\in \partial \Omega$ with dilation  
$0<\lambda_\xi\leq 1$, called the {\sl Denjoy--Wolff point}, such that  the sequence of iterates $(f^n)$ converges to $\xi$.
\end{teo}
This allows to partition the family of holomorphic self-maps of $\Omega$ as in the ball.
\begin{definition}
Let $\Omega\subset \C^q$ be a bounded strongly convex $C^3$ domain. 
A holomorphic self-map $f\colon \Omega\to \Omega$ is called {\sl elliptic} if it admits an interior fixed point. Otherwise it is called {\sl hyperbolic} if the dilation $\lambda_\xi$ at its  Denjoy--Wolff point satisfies $\lambda_\xi<1$, and it is called {\sl parabolic} if $\lambda_\xi=1$.
\end{definition}

\subsection{Squeezing function}
The squeezing function $S_\Omega\colon \Omega\to (0,1]$ of a bounded domain $\Omega\subset \mathbb C^q$  measures how much $\Omega$ resembles the ball $\B^q$.
\begin{definition}
Let $\Omega\subset \mathbb C^q$ be a bounded domain and $z\in\Omega$. If $\varphi:\Omega\rightarrow \B^q$ is an injective holomorphic function with $\varphi(z)=0$ we set
$$
S_{\Omega,\varphi}(z):=\sup\{r>0\colon B(0,r)\subset\varphi(\Omega)\},
$$
and
$$
S_\Omega(z):=\sup_\varphi\{S_{\Omega,\varphi}(z)\}.
$$
The function $S_\Omega:\Omega\rightarrow(0,1]$ is called the \emph{squeezing function} of the domain $\Omega$.
\end{definition}
By a normality argument it follows that the  sup is actually attained. 
\begin{prop}\label{maxattained}
\label{squeezing}
Let $\Omega\subset\mathbb C^q$ be a bounded domain and $z\in\Omega$. Then there exists an injective holomorphic map $\varphi\colon \Omega\rightarrow \B^q$, with $\varphi(z)=0$ such that 
$$
S_{\Omega,\varphi}(z)=S_\Omega(z).
$$
\end{prop}
We will need the following result proved in \cite{DGZ}.
\begin{teo}\label{deng}
If $\Omega\subset \C^q$ is a bounded strongly pseudoconvex domain with $C^2$ boundary, then 
$$\lim_{z\to\partial\Omega}S_{\Omega}(z)=1.$$
\end{teo}

\part{Forward iteration}

\section{Canonical Kobayashi hyperbolic semi-models}\label{canonicalforward}

In this section we construct a Canonical Kobayashi hyperbolic semi-model for a holomorphic self-map $f$ of a bounded domain $\Omega$, assuming that the squeezing function converges to 1 along  an  orbit. 
\begin{definition}
Let $X$ be a complex manifold and let $f\colon X\to X$ be a holomorphic self-map.
Let $x\in X$, and let $m\geq 1$. The {\sl forward $m$-step} $s_m(x)$ of $f$ at $x$ is the limit
$$s_m(x):=\lim_{n\to\infty}k_X(f^n(x), f^{n+m}(x)).$$ Such a limit exists since the sequence $(k_X(f^n(x), f^{n+m}(x)))_{n\geq 0}$ is non-increasing.
The {\sl divergence rate} $c(f)$ of $f$ is the limit 
\begin{equation}\label{defdivrate}
c(f):=\lim_{m\to\infty}\frac{k_X(f^m(x),x)}{m}.
\end{equation}
 It is shown in \cite{ArBr} that the limit above exists, does not depend on the point $x\in X$ and equals $\inf_{m\in \N}\frac{k_X(f^m(x),x)}{m}$.
\end{definition}
\begin{definition}\label{semimodel}
Let $X$ be a complex manifold and let $f\colon X\to X$ be a holomorphic self-map.
A {\sl semi-model} for $f$ is a triple  $(\Lambda,h,\v)$  where
$\Lambda$ is a complex manifold called the {\sl base space}, $h\colon X\to \Lambda$ is a holomorphic mapping, and $\v\colon \Lambda\to\Lambda$ is an automorphism such that
\begin{equation}
h\circ f=\v\circ h,
\end{equation}
and 
\begin{equation}\label{due}
\bigcup_{n\geq 0} \v^{-n}(h(X))=\Lambda.
\end{equation}

Let $(Z,\ell,\tau)$ and  $(\Lambda, h, \v)$ be two semi-models for the map $f$. A {\sl morphism of semi-models} 
$\hat\eta\colon (Z,\ell,\tau)\to (\Lambda, h, \v)$ is given by a  holomorphic map $\eta: Z\to \Lambda$ 
such that  the following diagram commutes:
\SelectTips{xy}{12}
\[ \xymatrix{X \ar[rrr]^h\ar[rrd]^\ell\ar[dd]^f &&& \Lambda \ar[dd]^\varphi\\
&& Z \ar[ru]^\eta \ar[dd]^(.25)\tau\\
X\ar'[rr]^h[rrr] \ar[rrd]^\ell &&& \Lambda\\
&& Z \ar[ru]^\eta}
\]
If the mapping $\eta \colon Z\to \Lambda$ is a biholomorphism, then we say that $\hat\eta\colon  (Z,\ell,\tau)\to (\Lambda, h, \v)$ is an {\sl isomorphism of semi-models}. Notice that then $\eta^{-1}\colon \Lambda\to Z$ induces a morphism $ {\hat\eta}^{-1}\colon  (\Lambda, h, \v)\to  (Z,\ell,\tau). $
\end{definition}

\begin{definition} 
Let $X$ be a complex manifold and let $f\colon X\to X$ be a holomorphic self-map. Let $(Z, \ell,\tau)$ be a semi-model for $f$  whose base space $Z$ is Kobayashi hyperbolic. We say that  $(Z, \ell,\tau)$ is a {\sl canonical Kobayashi hyperbolic semi-model} for $f$  if for any semi-model
$(\Lambda, h,\varphi )$ for $f$ such that the base space $\Lambda$ is Kobayashi hyperbolic,  there exists a unique   morphism of semi-models $\hat\eta\colon (Z, \ell,\tau)\to (\Lambda, h,\varphi )$.
 \end{definition}
 \begin{remark}\label{tango}
If $(Z, \ell,\tau)$ and $(\Lambda, h,\varphi )$ are two canonical Kobayashi hyperbolic semi-models for $f$, then they are isomorphic.
\end{remark}

In this section we prove the following  result.
\begin{teo}\label{astratto}
Let $\Omega\subset \C^{q}$ be a bounded domain, $f\colon \Omega\to \Omega$ be a holomorphic self-map and assume that there exists an orbit $(z_m)$ with $S_\Omega(z_m)\to 1$.
Then there exists a canonical Kobayashi hyperbolic semi-model $(\B^k,  \ell,\tau)$ for $f$ with $0\leq k\leq q$. Moreover, the following holds: 
\begin{enumerate}
\item for all $n\geq 0$, 
$$\lim_{m\to \infty}(f^m)^*k_\Omega=(\tau^{-n}\circ \ell)^*k_{\B^k},$$
\item the divergence rate of $\tau$ satisfies 
$$
c(\tau)=c(f)=\lim_{m\to \infty}\frac{s_m(x)}{m}=\inf_{m\in \N}\frac{s_m(x)}{m}.
$$
\end{enumerate}
\end{teo}
\begin{remark}
By Theorem \ref{deng} the assumptions of the theorem are satisfied when $\Omega$ is bounded strongly pseudoconvex with $C^2$ boundary and $(z_m)$ converges to $\partial\Omega$. 
\end{remark}
The proof of Theorem \ref{astratto} is based on  the following result.
\begin{prop}\label{core}
Let $\Omega\subset \C^{q}$ be a bounded domain, $f\colon \Omega\to \Omega$ be a holomorphic self-map and assume that there exists an orbit $(z_m)$ such that $S_\Omega(z_m)\to 1$.
Then there exists a family of holomorphic maps $(\alpha_n\colon \Omega\to Z)$, where $Z$ is an holomorphic retract of $\B^q$, such that the following hold:
\begin{itemize}
\item[(a)] for all $m\ge n\ge 0$,
$$
\alpha_m\circ f^{m-n}=\alpha_n,
$$
\item[(b)]  for every $n\ge 0$ we have $\alpha_n(\Omega)\subset\alpha_{n+1}(\Omega)$ and
\begin{equation}\label{therewolf}
\bigcup_{n\in\mathbb N}\alpha_n(\Omega)=Z,
\end{equation}
\item[(c)]  for all $n\ge 0$,
\begin{equation}\label{ulula}
\lim_{m\to\infty}  (f^m)^*\, k_\Omega=\alpha_n^* \, k_Z,
\end{equation}
\item[(d)]  {\bf Universal property}: let $Q$ be a Kobayashi hyperbolic complex manifold and $(\gamma_n\colon \Omega\to Q)$ a family of holomorphic mappings satisfying $\gamma_m\circ f^{m-n}=\gamma_n$ for all  $m\geq n\geq 0$. Then there exists a unique holomorphic map $\Gamma\colon Z\to Q$ such that $\gamma_n=\Gamma\circ \alpha_n$ for all $n\geq 0$.
\end{itemize}
\end{prop}
\begin{remark}
Such family $(\alpha_n)$ is a canonical Kobayashi hyperbolic direct limit for the sequence of iterates $(f^{m-n}\colon\Omega\to\Omega)$, see \cite[Definition 2.7]{Ar}.
\end{remark}
Once Proposition \ref{core} is proved, the proof of Theorem \ref{astratto} is the same as that of \cite[Theorem 4.6]{Ar}. We  present a sketch of the construction of the semi-model for the convenience of the reader.

\begin{proof}[Proof of Theorem \ref{astratto}]
Define $\ell:=\alpha_0$ and $\gamma_n:=\alpha_n\circ f$. It is not hard to show that $(\gamma_n\colon\Omega\rightarrow Z)$ is a family of holomorphic mappings satisfying $\gamma_m\circ f^{m-n}=\gamma_n$ for all $m\ge n\ge 0$. Therefore by the universal property of the family $(\alpha_n)$ there exists a unique holomorphic map $\tau\colon Z\rightarrow Z$ such that for all $n\ge 0$,
$$
\tau\circ\alpha_n=\gamma_n=\alpha_n\circ f,
$$
in particular $\tau\circ \ell=\ell\circ f$.

Similarly if we define $\tilde\gamma_n:=\alpha_{n+1}$ we obtain a holomorphic map $\delta:Z\rightarrow Z$ such that $\tilde\gamma_n=\delta\circ \alpha_n$ for all $n\ge 0$. It is easy to see that
$$
\tau\circ\delta\circ\alpha_n=\delta\circ\tau\circ\alpha_n=\alpha_n.
$$
By the universal property of the family $(\alpha_n)$ described in Proposition \ref{core}, we conclude that $\delta=\tau^{-1}$, proving that $\tau$ is an automorphism of $Z$.
Since for all $n\geq 0$, $$\tau^n\circ \alpha_n=\alpha_n\circ f^n=\ell,$$ it follows that $\alpha_n=\tau^{-n}\circ \ell$.
The triple $(Z,\ell,\tau)$ is a semi-model thanks to \eqref{therewolf} Notice that $Z$ being a holomorphic retract of $\B^q$, it is biholomorphic to a ball of dimension $0\le k\le q$. 
The universal property of the canonical Kobayashi hyperbolic semi-model $(Z,\ell,\tau)$ is a direct consequence of the universal property of the family $(\alpha_n)$. 

Point (1) follows immediately from (\ref{ulula}), hence we are left with proving  (2). From \cite[Proposition 2.7]{ArBr} it follows that $c(f)=\lim_{m\to \infty}\frac{s_m(x)}{m}=\inf_{m\in \N}\frac{s_m(x)}{m}$.  Equation (\ref{ulula}) immediately gives  the following formula for the forward $m$-step, which implies $(2)$,
$$s_m(x)=k_Z(\alpha_0(x),\alpha_0\circ f^m(x))=k_Z(\ell(x), \tau^{m}\circ\ell(x)).$$

\end{proof}

The proof of Proposition \ref{core} is articulated  in several intermediate lemmas. Let $(z_m)$ be an $f$-orbit in $\Omega$ with $S_\Omega(z_m)\to 1$. By Proposition \ref{maxattained} there exists a sequence $(\psi_m\colon \Omega\rightarrow \B^q)$ of holomorphic injective maps with $\psi_m(z_m)=0$ and
$$
\psi_m(\Omega)\supset B(0,S_\Omega(z_m)).
$$
Notice that for every compact subset $K\subset \B^q$ the inverse map $\psi_m^{-1}$ is defined on $K$ for $m$ sufficiently large.

Since $\psi_{m}\circ f^m(z_0)=0$ for all $m\ge 0$ and since $\B^q$ is taut, there exists a subsequence $(m_0(h))$ such that the sequence $(\psi_{m_0(h)}\circ f^{m_0(h)} )$ converges uniformly on compact subsets to a holomorphic map $\alpha_0\colon \Omega\to \B^q$ and $\alpha_0(z_0)=0$. Similarly, there exists a subsequence  $(m_1(h))$ of $(m_0(h))$ such  that the sequence $(\psi_{m_1(h)}\circ f^{m_1(h)-1} )$ converges uniformly on compact subsets to a holomorphic map $\alpha_1\colon \Omega\to \B^q$ and $\alpha_1(z_1)=0$.
Iterating this procedure we obtain a family of subsequences $\{(m_n(h))_{h\geq 0}\}_{n\geq 0}$ and a family of holomorphic maps
$$(\alpha_n\colon \Omega\to \B^q)_{n\geq 0}$$ 
such that  $$\psi_{m_n(h)}\circ f^{m_n(h)-n}  \stackrel{h\to\infty}\longrightarrow \alpha_n$$ uniformly on compact subsets and $\alpha_n(z_n)=0$.
Notice that for all $m\geq n\geq 0$,
\begin{equation}\label{jonsnow}
\alpha_{m}\circ f^{m-n}=\alpha_{n}.
\end{equation}
Let $\nu(h):=m_h(h)$ be the diagonal subsequence, which for all $j\geq 0$ is eventually a subsequence of $(m_j(h))_{h\geq 0}.$

Consider the sequence $\beta_{\nu(h)}:=\alpha_{\nu(h)}\circ \psi_{\nu(h)}^{-1}$. Given a compact subset $K\subset \B^q$ and $h$ large enough, the map $\beta_{\nu(h)}$ is well defined on $K$ and $\beta_{\nu(h)}(K)\subset \B^q$. Notice that  $\beta_{\nu(h)}(0)=0$  for all $h\geq 0$. By the tautness of $\B^q$ up to extracting a further subsequence of $\nu(h)$ we have that the sequence $(\beta_{\nu(h)})$ converges uniformly on compact subsets to a holomorphic map $\alpha\colon \B^q\to \B^q.$

\begin{lemma}
For all $j\geq 0$,
\begin{equation}\label{stabilizza}
\alpha\circ \alpha_j=\alpha_j.
\end{equation}
\end{lemma}
\begin{proof}
Let $z\in \Omega$. For all positive integers $h$ such that $\nu (h)\geq j$, we have, using (\ref{jonsnow}),
$$\alpha_j(z)=(\alpha_{\nu(h)}\circ f^{\nu(h)-j})(z)=(\alpha_{\nu(h)}\circ \psi_{\nu(h)}^{-1}\circ \psi_{\nu(h)}\circ f^{\nu(h)-j}) (z)\stackrel{h\to\infty}\longrightarrow (\alpha\circ \alpha_j)(z).$$
\end{proof}
\begin{lemma}
The map $\alpha\colon  \B^q\to \B^q$ is a holomorphic retraction, that is
$$\alpha\circ \alpha=\alpha.$$
\end{lemma}
\begin{proof}
Let $z\in \B^q$.
From (\ref{stabilizza}) 
 we get, for all $h\geq 0$ big enough,
$$(\alpha\circ\beta_{\nu(h)})(z)=(\alpha\circ \alpha_{\nu(h)}\circ\psi_{\nu(h)}^{-1})(z)=(\alpha_{\nu(h)}\circ\psi_{\nu(h)}^{-1})(z)=\beta_{\nu(h)}(z),$$
and the result follows since $\beta_{\nu(h)}\to \alpha$.
\end{proof}
Define $Z:=\alpha(\B^q)$. Being a holomorphic retract, it is a closed complex submanifold of $\B^q$, biholomorphic to a $k$-dimensional ball $\B^k$, with $0\leq k\leq q$.
By \eqref{stabilizza} it follows that, for all $j\geq 0$, $$\alpha_j(\Omega)\subset Z.$$

Let $(A, \Lambda_{n})$ be the direct limit of the dynamical system $(f^n\colon \Omega\to\Omega)$. Recall that
$A:=(\Omega\times \mathbb N)/_\sim,$
where $(x,n)\sim(y,u)$ if and only if $f^{m-n}(x)=f^{m-u}(y)$ for $m$ large enough, and  the equivalence class of $(x,n)$ is denoted by $[x,n]$. The map $\Lambda_n:\Omega\to A$ is defined by $\Lambda_n(x)=[x,n]$.

By the universal property of direct limits, there exists a unique map $\Psi:A\to Z$
such that, for all $n\geq 0$,
$$
\alpha_{n}=\Psi\circ\Lambda_{n}.
$$
The mapping $\Psi$ sends the point $[x,n]\in A$ to $\alpha_n(x)$. Define on $A$ the following equivalence relation:
$$
[x,n]\simeq [y,u]\quad \iff \quad k_{\Omega}(f^{m-n}(x),f^{m-u}(y))\stackrel{m\to \infty}\longrightarrow 0.
$$
	
\begin{lemma}\label{grosso}
The map $\Psi\colon A\to Z$ is surjective and $\Psi([x,n])=\Psi([y,u])$ if and only if $[x,n]\simeq [y,u]$.
\end{lemma}

\begin{proof}
We first prove surjectivity. We have to prove that for all $z\in Z$ there
exists $x\in \Omega$ and $n\geq 0$  such that  $\alpha_{n}(x)=z$.
Let $U\subset Z$ be a relatively compact neighborhood in $Z$ of $z$.
The sequence  $(\beta_{\nu(h)}|_U)=(\alpha_{\nu(h)}\circ\psi_{\nu(h)}^{-1}|_{U})$ is well defined for $h$  big enough and converges uniformly to $\alpha|_U={\sf id}_U$, and therefore is eventually injective and its image eventually contains $z$.

If $[x,n]\simeq [y,u]$, then since the Kobayashi distance is non-expansive with respect to holomorphic maps, we have
$$k_Z(\Psi[x,n], \Psi[y,u])=k_{Z}(\alpha_{m}\circ f^{m-n}(x),\alpha_{m}\circ f^{m-u}(y))\leq k_{\Omega}(f^{m-n}(x),f^{m-u}(y))\stackrel{m\to \infty}\longrightarrow 0.$$
As $Z$ is Kobayashi hyperbolic, it follows that $\Psi[x,n]= \Psi[y,u]$.

Conversely, assume that $\Psi([x,n])=\Psi([y,u])$,  and fix $j\geq \max\{n,u\}$. We have
$$\alpha_{j}\circ f^{j-n}(x)=\alpha_{j}\circ f^{j-u}(y).$$
By definition of the map $\alpha_{j}$ it follows that 
$$\lim_{h\to\infty} \psi_{m_j(h)}\circ f^{m_j(h)-n}(x)=\lim_{h\to\infty} \psi_{m_j(h)}\circ f^{m_j(h)-u}(y).$$
We claim that this implies that $[x,n]\simeq [y,u]$.
Notice that, since the sequence
$$\big(k_\Omega(f^{m-n}(x), f^{m-u}(y))\big)_{m\geq \max\{n,u\}}$$ is decreasing, 
 it suffices to show that $$k_{\Omega}(f^{m_j(h)-n}(x),f^{m_j(h)-u}(y))\stackrel{h\to \infty}\longrightarrow 0.$$
Denote $z_h:= \psi_{m_j(h)}\circ f^{m_j(h)-n}(x)$ and $w_h:= \psi_{m_j(h)}\circ f^{m_j(h)-u}(y)$. Then the sequences $(z_h)$ and $(w_h)$ converge to the same point $a\in \B^q$. 

Let $B\subset \subset \B^q$ be a ball centered in $a$. When $h$ is sufficiently large we have that $\psi_{m_j(h)}^{-1}\colon B\rightarrow\Omega$ is well defined and since the Kobayashi distance is not expanding, we conclude that
$$
k_{\Omega}(f^{m_j(h)-n}(x),f^{m_j(h)-u}(y))\le k_{B}(z_h,w_h)\to 0.
$$
 \end{proof}

\begin{lemma}\label{x4}
For all $n\geq 0$,
\begin{equation*}
\lim_{m\to\infty}  (f^m)^*\, k_\Omega=\alpha_n^* \, k_Z.
\end{equation*}
\end{lemma}
\begin{proof}
For all $m\geq n\ge 0$ we have $$k_Z(\alpha_n(x),\alpha_n(y))=k_Z(\alpha_m\circ f^{m-n}(x),\alpha_m\circ f^{m-n}(y))\leq k_\Omega(f^{m-n}(x), f^{m-n}(y)).$$
Hence $k_Z(\alpha_n(x),\alpha_n(y))\leq \lim_{m\to\infty}k_\Omega(f^{m}(x), f^{m}(y))$.

To obtain the inverse inequality, denote $z_h:= \psi_{m_n(h)}\circ f^{m_n(h)-n}(x)$ and $w_h:= \psi_{m_n(h)}\circ f^{m_n(h)-n}(y)$. Then  $(z_h)$ converges to $\alpha_n(x)$ and $(w_h)$ converges to  $\alpha_n(y)$. 
Fix $\epsilon>0$, then there exist a ball $B=B(0,r)\subset \B^q$, with radius close enough to $1$ such that it contains both $\alpha_n(x),\alpha_n(y)$, and  such that for some $h_0\ge0$ we have 
$$
k_{B}(z_h, w_h)\leq k_{\B^q}(\alpha_n(x),\alpha_n(y))+\epsilon,\qquad \forall h\ge h_0.
$$
Let $h_1\geq 0$ be such that for all $h\geq h_1$ we have that $\psi_{m_n(h)}(\Omega) \supset B$. Then for all $h\geq \max\{h_0,h_1\}$ we have 
$$
k_{\Omega}(f^{m_n(h)-n}(x), f^{m_n(h)-n}(y))\leq k_B(z_h,w_h)\le k_{\B^q}(\alpha_n(x),\alpha_n(y))+\epsilon.$$
proving that for every $\varepsilon>0$ we have
$$
\lim_{m\to\infty}k_{\Omega}(f^m(x), f^{m}(y))\le k_{Z}(\alpha_n(x),\alpha_n(y))+\epsilon,
$$
where we used the fact that $\alpha_n(x),\alpha_n(y)\in Z$ and that $k_{\B^q}|_Z=k_Z$.
\end{proof}

We are now ready to prove  Proposition \ref{core}. Points $(a)$ and $(c)$ correspond precisely to \eqref{jonsnow} and Lemma \ref{x4}. By \eqref{jonsnow} it is clear that $\alpha_n(\Omega)=\alpha_{n+1}( f(\Omega))\subset \alpha_{n+1}(\Omega)$, and  by Lemma \ref{grosso} we obtain that the union of the sets $\alpha_n(\Omega)$ coincides with $Z$, which proves point $(b)$.

It remains to prove the universal property $(d)$. Let $Q$ be a Kobayashi hyperbolic complex manifold and let $(\gamma_n\colon \Omega\to Q)$ be a family of holomorphic maps satisfying  $\gamma_m\circ f^{m-n}=\gamma_n$ for all  $m\geq n\geq 0$. By the universal property of the direct limit,  there exists a unique map $\Phi\colon A\to Q$ such that $\gamma_n=\Phi\circ \Lambda_n$ for all $n\geq 0$. The map $\Phi$ passes to the quotient to a map  $\hat \Phi\colon A/_\simeq\to Q$.
Indeed, if $[(x,n)]\simeq [(y,u)]$, for all $m\geq n,u$ we have that $\Phi([(x,n)])=\gamma_m\circ f^{m-n}(x)$ and $\Phi([(y,u)])=\gamma_m\circ f^{m-u}(y)$. Hence $$k_{Q}(\Phi([(x,n)]),\Phi([(y,u)]))\leq k_\Omega(f^{m-n}(x),f^{m-u}(y))\stackrel{m\rightarrow\infty} \longrightarrow 0,$$
 and thus $\Phi([(x,n)])=\Phi([(y,u)])$.
Set $$\Gamma\coloneqq \hat\Phi\circ {\hat\Psi}^{-1}\colon Z\to Q.$$
The mapping $\Gamma $ acts in the following way: if $z\in Z$, then  
there exists $x\in \Omega$ and $n\geq 0$ such that $\alpha_n(x)=z$, and then $\Gamma(z)=\gamma_n(x).$ It is thus clear that it is the unique map satisfying $\Gamma\circ \alpha_n=\gamma_n$ for all $n\geq 0$. The map $\Gamma$ is holomorphic. Indeed, if $z\in Z$,  by the proof of Lemma \ref{grosso}, there exist a neighborhood $U$ of  $z$ in $Z$, a point $w\in U$ and $m'\geq 0$  such that $(\alpha_{m'}\circ\psi_{m'}^{-1}|_U\colon U\to Z)$ is defined, holomorphic and  injective and $\alpha_{m'}\circ\psi_{m'}^{-1}(w)=z$.   Thus there exists an open neighborhood $V\subset Z$ of $z$ and a holomorphic map $\sigma\colon V\to U$ such that 
 $$\alpha_{m'}\circ\psi_{m'}^{-1}\circ\sigma={\sf id}_V.$$
 Then, for all $y\in V$,
  $$\Gamma(y)=\Gamma\circ\alpha_{m'}\circ\psi_{m'}^{-1}\circ\sigma(y)=\gamma_{m'}\circ\psi_{m'}^{-1}\circ\sigma(y),$$
that is, $\Gamma$ is holomorphic in $V$, which concludes the proof of $(d)$ and of Proposition \ref{core}.

\section{Main result on strongly convex domains}

In this section we apply the results of the previous section to the case of strongly convex domains, and we prove Theorem~\ref{teoconvessoavanti}. We start with the following proposition in which we compare the dilation $\lambda_{\xi}$ with the divergence rate $c(f)$.

 \begin{prop}\label{ugualedil}
Let $\Omega\subset \C^n$ be a bounded strongly convex domain  with $C^3$ boundary. Let $f\colon \Omega\to \Omega$ be a holomorphic self-map without interior fixed points, and let $\xi$ be its Denjoy--Wolff point. Then 
$$\log \lambda_\xi=-c(f).$$
\end{prop}
\begin{proof}
Let $p,z\in \Omega$. We have
\begin{align*}
-c(f)&=\lim_{n\to +\infty}\frac{-k_\Omega(p,f^n(z))}{n}\\&\geq \liminf_{n\to+\infty}[-k_\Omega(p,f^{n+1}(z))+k_\Omega(p,f^n(z))]\\&\geq \liminf_{z\to \xi}[k_\Omega(p,z)-k_\Omega(p, f(z))],
\end{align*}
where we used that for all real sequences $(a_n)$
$$\liminf_{n\to+\infty}\frac{a_n}{n}\geq \liminf_{n\to +\infty}[a_{n+1}-a_n].$$
Hence $\log \lambda_\xi\leq -c(f).$

If $\lambda_\xi=1$, the result follows. If $0<\lambda_\xi<1$, we obtain the converse inequality in the following way.
We claim that if $z\in E_\Omega(p,\xi,R)$, then $k_\Omega(p,z)\geq -\log R.$ Indeed we have that $k_\Omega(p,z)\geq k_\Omega(p,w)-k_\Omega(z,w)$ for all $w\in \Omega$ and thus
$$k_\Omega(p,z)\geq \lim_{w\to \xi} [k_\Omega(p,w)-k_\Omega(z,w)]>-\log R.$$
Let $z\in E_\Omega(p,\xi, 1)$. 
It follows from Proposition \ref{julialem} that $f^n(z)\in E_\Omega(p,\xi, \lambda_\xi^n)$. Hence
$$\frac{k_\Omega(p, f^n(z))}{n}\geq -\frac{\log \lambda_\xi^n}{n}=-\log \lambda_\xi.$$

\end{proof}

\begin{remark}
 Proposition  \ref{ugualedil} shows why the concept of divergence rate is relevant in this context. Indeed, let $f\colon \Omega\to \Omega$ be a holomorphic self-map without interior fixed points, and let $(\B^k,\ell,\tau)$ be the canonical Kobayashi  hyperbolic semi-model given by Theorem \ref{astratto}. 
Assume that $\tau$ has no interior fixed points.  Since $c(f)=c(\tau)$ it follows  that the dilation of $f$ and $\tau$ at their respective Denjoy-Wolff points is the same. 
\end{remark}

%

We are now ready to state the main result of this section.
 \begin{teo}\label{fin1}
 Let $\Omega\subset  \C^q$ be a bounded strongly convex domain with $C^3$ boundary.
Let $f\colon \Omega\to \Omega$ be a hyperbolic holomorphic self-map, with Denjoy--Wolff point $\xi$.
Then there exist
\begin{enumerate}
\item an integer $k$ such that $1\leq k\leq q$, 
\item a hyperbolic automorphism $\tau\colon \HH^k\to \HH^k$ of the form
\begin{equation}\label{lomo1}
\tau(z_1,z')= \left(\frac{1}{\lambda_\xi} z_1,\frac{e^{it_1}}{\sqrt \lambda_\xi}z'_1,\dots, \frac{e^{it_{k-1}}}{\sqrt \lambda_\xi}z'_{k-1} \right),
\end{equation}
where $t_j\in \R$ for $1\leq j\leq k-1$, 
\item a holomorphic mapping
$\ell\colon \Omega\to \HH^k$,
\end{enumerate}
 such that the triple
 $(\HH^k,\ell,\tau)$ is a  canonical Kobayashi hyperbolic semi-model for $f$.
\end{teo}
 \begin{proof}
Thanks to Theorem \ref{astratto} we obtain the existence of a canonical Kobayashi hyperbolic semi-model $(\HH^k,\ell, \tau)$ for $f$ with $c(\tau)=c(f)>0$. It immediately follows that $k>0$. Moreover,  by Proposition \ref{ugualedil} it follows that $\tau$ is a hyperbolic automorphism of $\HH^k$ with dilation  $\lambda_\xi$ at its  Denjoy--Wolff point. Now we can change variables in $\HH^k$ to put $\tau$ the form (\ref{lomo1}), concluding the proof.
\end{proof}
\begin{remark}
It is natural to ask  whether  
\begin{equation}\label{manca}
K\hbox{-}\lim_{z\to \xi}h(z)=\infty,
\end{equation}
which is the case when $\Omega=\B^q$.
If there exists an orbit $(z_n)$ which enters eventually a Koranyi  region with vertex at the Denjoy-Wolff point $\xi$, then 
 (\ref{manca}) follows as in \cite[Theorem 5.6]{ArBr}. Notice that the proof of \cite[Theorem 5.6]{ArBr} is given for the ball $\B^q$, but a similar proof works  for bounded strongly convex domains with smooth boundary.
 If $\Omega=\B^q$ then all orbits eventually enter a Koranyi region.
  It is an open question whether such an orbit $(z_n)$  exists when $\Omega$ is a bounded strongly convex domain with $C^3$ boundary.
\end{remark}

We end this section giving a similar result for parabolic nonzero-step maps.
\begin{definition}
 Let $\Omega\subset\subset  \C^q$ be a  strongly convex domain with $C^3$ boundary.
If $f\colon \Omega\to \Omega$ is a parabolic holomorphic self-map, we say that it is {\sl nonzero-step} if for all $z\in \Omega$ we have that
$s_1(z)>0$.

\end{definition}

\begin{teo}\label{parabolicresult}
 Let $\Omega\subset \C^q$ be a bounded strongly convex domain with $C^3$ boundary.
Let $f\colon \Omega\to \Omega$ be a parabolic nonzero-step holomorphic self-map with  Denjoy--Wolff point $\xi$.
Then there exist
\begin{enumerate}
\item an integer $k$ such that $1\leq k\leq q$, 
\item a parabolic automorphism $\tau\colon \HH^k\to \HH^k$ of the form
\begin{equation}\label{lomo2}
\tau(z_1,z')=(z_1\pm1,e^{it_1}z'_1,\dots e^{it_{k-1}}z'_{k-1}),
\end{equation}
where $t_j\in \R$ for $1\leq j\leq k-1$, or of the form
\begin{equation}\label{lomo3}
\tau(z_1,z')=(z_1-2z'_1+i, z'_1-i,e^{it_2}z'_2,\dots e^{it_{k-1}}z'_{k-1}),
\end{equation}
 where  $t_j\in \R$ for $2\leq j\leq k-1,$
\item a holomorphic mapping
$\ell\colon \B^q\to \HH^k$,
\end{enumerate}
 such that the triple
 $(\HH^k,\ell,\tau)$ is a  canonical Kobayashi hyperbolic semi-model for $f$.
\end{teo}
\begin{proof}
Let $(\HH^k,\ell,\tau)$ be the canonical Kobayashi hyperbolic model for $f$ given by Theorem \ref{astratto}.
Then for all $z\in \HH^k$, it follows from (\ref{ulula}) that  $k_{\HH^k}(z,\tau(z))=s_1(z)>0.$
Hence $k\geq 1$, and $\tau$ is not elliptic. Moreover, since  $c(\tau)=c(f)$, it follows from Proposition \ref{ugualedil} that  $\tau$ is parabolic. Finally we can change holomorphic coordinates and put $\tau$ in the form (\ref{lomo2}) or  (\ref{lomo3}).
\end{proof}

\part{Backward iteration}
%
%
%
\section{Canonical pre-models}\label{canonicalbackward}

In this section we construct a canonical pre-model for a holomorphic self-map $f$ of a bounded taut domain $\Omega$ 
assuming the existence of
a backward orbit with bounded step along which the squeezing function converges to 1.

\begin{definition}
Let $X$ be a complex manifold and let $f\colon X\to X$ be a holomorphic self-map.
Let $\beta=(x_m)_{m\geq 0}$ be a backward orbit for $f$, meaning that $f(x_{m+1})=x_m$ for all $m\geq 0$. The {\sl backward $m$-step} $s_m(\beta)$ of $f$ at $\beta$ is the limit
$$s_m(\beta):=\lim_{n\to\infty}k_X(x_n, x_{n+m}).$$ Such a limit exists since the sequence $(k_X(x_n, x_{n+m}))_{n\geq 0}$ is non-decreasing. We will say that the backward orbit has bounded step if $s_1(\beta)<\infty$.
If $(y_m)$ is a backward orbit for $f$, we denote by $[y_m]$ the family of all backward orbits $(z_m)$ of $f$ such that the sequence
$(k_X(z_m,y_m))$ is bounded.
\end{definition}
\begin{definition}\label{pre-model}
Let $X$ be a complex manifold and let $f\colon X\to X$ be a holomorphic self-map.
A {\sl pre-model} for $f$ is a triple  $(\Lambda,h,\v)$  where
$\Lambda$ is a complex manifold called the {\sl base space}, $h\colon \Lambda\to X$ is a holomorphic mapping, and $\v\colon \Lambda\to\Lambda$ is an automorphism such that
\begin{equation}
f\circ h=h\circ\v.
\end{equation}
Let  $(\Lambda,h,\v)$ and $(Z,\ell,\tau)$  be two pre-models for $f$. A  {\sl morphism of pre-models} $\hat\eta\colon  (\Lambda,h,\v) \to (Z,\ell,\tau) $ is given by a holomorphic mapping  $\eta \colon \Lambda\to Z$ such that the following diagram commutes:
\SelectTips{xy}{12}
\[ \xymatrix{\Lambda \ar[rrr]^h\ar[rd]^\eta\ar[dd]^\v &&& X \ar[dd]^f\\
& Z \ar[rru]^\ell \ar[dd]^(.25)\tau\\
\Lambda\ar'[r][rrr]^(.25)h \ar[rd]^\eta &&& X\\
& Z \ar[rru]^\ell}
\]
If the mapping $\eta \colon \Lambda\to Z$ is a biholomorphism, then we say that $\hat\eta\colon (\Lambda,h,\v) \to (Z,\ell,\tau) $ is an {\sl isomorphism of pre-models}. Notice then that $\eta^{-1}\colon Z\to \Lambda$ induces a morphism $ {\hat\eta}^{-1}\colon  (Z,\ell,\tau)\to   (\Lambda,h,\v). $
\end{definition}
\begin{definition}\label{definizionecanprem}
Let $X$ be a complex manifold, let $f\colon X\to X$ be a holomorphic self-map and let  $(\Lambda,h,\v)$ be a pre-model for $f$.
If  $[y_m]$ is a class of backward orbits, we say that $(\Lambda,h,\v)$ is {\sl associated} with $[y_m]$ if for some (and hence for any) $x\in X$ we have that  $(h\circ \v^{-m}(x))\in [y_m]$.

We say that $(Z, \ell,\tau)$ is a  {\sl canonical pre-model}  for $f$  associated with $[y_m]$ if
\begin{enumerate}
\item  $(Z, \ell,\tau)$ is a pre-model for $f$ associated with $[y_m]$, and
\item for any other pre-model
$(\Lambda,h,\v)$ for $f$ associated with $[y_m]$  there exists a unique morphism of pre-models $\hat\eta\colon (\Lambda,h,\v) \to (Z,\ell,\tau) $.
\end{enumerate}
\end{definition}

\begin{remark}\label{boundedstepremark}
If $(Z, \ell,\tau)$ and $(\Lambda,h,\v)$ are two canonical pre-models for $f$ associated with the same class $[y_m]$, then they are isomorphic. 
Moreover it is is easy to see (see e.g.  \cite[Lemma 7.4]{Ar}) that if a pre-model is associated with a class $[y_m]$ then every backward orbit in $[y_m]$ has bounded step.
\end{remark}

Let $\Omega\subset\subset  \C^{q}$ be a  taut domain and $f\colon \Omega\to \Omega$ be a holomorphic self-map. Let $(\Theta,V_n)$ be the inverse limit of the sequence of iterates  $(f^{m-n}\colon \Omega\to\Omega)$. Recall that  $\Theta$ is defined as
$$
\Theta:=\{(z_m)_{m\ge 0}\in\Omega^\N\colon (z_m)\textrm{ is a backward orbit for $f$}\},
$$
and the map $V_n\colon\Theta\to \Omega$ is defined as $V_n((z_m))=z_n$. 

The following theorem is the analogous of Theorem \ref{astratto} for the backward dynamics. Notice that here we assume that the domain $\Omega$ is taut. This condition will be used in the proof of Proposition \ref{backcore} in order to construct the sequence of maps $\alpha_n:Z\to\Omega$.
\begin{teo}\label{backastratto}
Let $\Omega\subset  \C^{q}$ be a bounded taut domain and $f\colon \Omega\to \Omega$ be a holomorphic self-map. Assume that there exists a backward orbit $(z_m)$ with bounded step and  $S_\Omega(z_m)\to 1$.

Then there exists a canonical pre-model $(\B^k,  \ell,\tau)$ for $f$ associated with $[z_m]$, where $0\leq k\leq q$. Moreover, the following holds: 
\begin{enumerate}
\item the image of the map $\ell$ is
$$\ell(\B^k)=V_0([z_m]),$$
and for all $(w_m)\in [z_m]$, there exists a unique $z\in \B^k$ such that $(\ell\circ\tau^{-m}(z))=(w_m),$
\item we have $$\lim_{m\to \infty}(\ell\circ\tau^{-m})^*k_\Omega=k_{\B^k},$$
\item if $\beta$ is a backward orbit in the class $[z_m]$, then the divergence rate of $\tau$ satisfies 
$$ c(\tau)=\lim_{m\to \infty}\frac{s_m(\beta)}{m}=\inf_{m\in \N}\frac{s_m(\beta)}{m}.$$
\end{enumerate}
\end{teo}
\begin{remark}
The assumptions of the theorem are satisfied when $\Omega$ is bounded strongly pseudoconvex with $C^2$ boundary and $(z_m)$ converges to $\partial\Omega$, since in this case the domain $\Omega$ is taut (see \cite[Corollary 2.1.14]{Ab}) and $S_\Omega(z_m)\to 1$ by Theorem \ref{deng}. 
\end{remark}
The proof of Theorem \ref{backastratto} is based on  the following result.
\begin{prop}\label{backcore}
Let $\Omega\subset  \C^{q}$ be a bounded taut domain and $f\colon \Omega\to \Omega$ be a holomorphic self-map. Assume that there exists a backward orbit $(z_m)$ with bounded step such that  $S_\Omega(z_m)\to 1$.

Then there exists a family of holomorphic maps $(\alpha_n\colon Z\to\Omega)$, where $Z$ is an holomorphic retract of $\B^q$, such that the following hold:

\begin{itemize}
\item[(a)] for all $m\ge n\ge 0$,
$$
\alpha_n=f^{m-n}\circ\alpha_m,
$$
\item[(b)] let $\Psi:Z\rightarrow \Theta$ be the map defined as $\Psi(z)=(\alpha_m(z))$. Then $\Psi$ is injective, $\Psi(0)=(z_m)$ and
 \begin{equation}\label{backtherewolf}
\Psi(Z)=[z_m],
\end{equation}
\item[(c)]\begin{equation}\label{backulula}
\lim_{m\to\infty} \alpha_m^* \, k_\Omega=\, k_Z,
\end{equation}
\item[(d)] {\bf Universal property}: let $Q$ be a complex manifold and $(\gamma_n\colon Q\to\Omega)$ a family of holomorphic mappings satisfying $\gamma_n=f^{m-n}\circ\gamma_m$ for all  $m\geq n\geq 0$, and such that $(\gamma_m(x))\in[z_m]$ for some (and hence for any) $x\in Q$. Then there exists a unique holomorphic map $\Gamma\colon Q\to Z$ such that $\alpha_m\circ \Gamma=\gamma_m$ for all $m\geq 0$.
\end{itemize}
\end{prop}
\begin{remark}
Such family $(\alpha_n)$ is a canonical  inverse limit for the sequence of iterates $(f^{m-n}\colon\Omega\to\Omega)$ associated with $[z_m]$, see \cite[Definition 6.9]{Ar}.
\end{remark}

Once Proposition \ref{backcore} is proved, the proof of Theorem \ref{backastratto} is the same as the proof of \cite[Theorem 8.7]{Ar}. We  present a sketch of the construction of the pre-model for the convenience of the reader.
\begin{proof}[Proof of Theorem \ref{backastratto}]
Following the proof of \cite[Theorem 8.7]{Ar} we define $\ell:=\alpha_0$ and $\gamma_n:=f\circ \alpha_n$. It is not hard to show that $(\gamma_n\colon Z\rightarrow \Omega)$ is a family of holomorphic mappings satisfying $\gamma_n=f^{m-n}\circ\gamma_m$ for all $m\ge n\ge 0$. Furthermore since $(z_m)$ has bounded step, we have 
$$
\sup_m k_\Omega(\gamma_m(0),z_m)=\sup_m k_\Omega(z_{m-1},z_m)<\infty.
$$
It follows that $(\gamma_m(x))\in[z_m]$ for every $x\in Q$. 

By the universal property of the family $(\alpha_n)$ there exists a unique holomorphic map $\tau:Z\rightarrow Z$ such that for all $n\ge 0$,
\begin{equation}\label{backcomm}
\alpha_n\circ\tau=\gamma_n= f\circ\alpha_n,
\end{equation}
in particular $\ell\circ\tau= f\circ \ell$.

Similarly if we define $\tilde\gamma_n:=\alpha_{n+1}$ we obtain a holomorphic map $\delta:Z\rightarrow Z$ such that $\tilde \gamma_n=\alpha_n\circ\delta$ for all $n\ge 0$. It is easy to see that, for all $n$,
$$
\alpha_n\circ\tau\circ\delta=\alpha_n\circ\delta\circ\tau=\alpha_n.
$$
By the universal property of the family $(\alpha_n)$ described in Proposition \ref{backcore}, we conclude that $\delta=\tau^{-1}$, proving that $\tau$ is an automorphism of $Z$. 
For all $n\geq 0$ we have that 
$$\alpha_n\circ \tau^n=f^n\circ \alpha_n=\ell,$$
and thus $\alpha_n=\ell\circ \tau^{-n}$ for all $n\geq 0$.
The triple $(Z,\ell,\tau)$ is a pre-model associated with $[z_m]$  thanks to \eqref{backtherewolf}.
Since $Z$ is a holomorphic retract of $\B^q$, it is biholomorphic to a ball of dimension $0\le k\le q$.

The universal property of the canonical pre-model $(Z,\ell,\tau)$ is  a direct consequence of the universal property of the family $(\alpha_n)$.
Since $\ell=V_0\circ \Psi$, point $(1)$ is an immediate consequence of \eqref{backtherewolf}.
Point (2) easily follows from (\ref{backulula}).


Finally, let $\beta:=(w_m)$ be a backward orbit in the class $[z_m]$. By \eqref{backtherewolf} there exists $z\in Z$ so that $(w_m)=(\alpha_m(z))$. Using again \eqref{backulula}, we obtain the following formula for the backward $m$-step of $\beta$, which directly implies $(3)$ 
\begin{align*}
s_m(\beta)=\lim_{n\to\infty}k_\Omega(\alpha_{n-m}(z),\alpha_{n}(z))=\lim_{n\to\infty}k_\Omega(\alpha_n\circ\tau^{m}(z),\alpha_n(z))=k_Z(z,\tau^m(z)).
\end{align*}
 
\end{proof}

The proof of Proposition \ref{backcore} is articulated  in several intermediate lemmas. Let $(z_m)$ be a backward orbit with bounded step satisfying $S_\Omega(z_m)\to 1$. Let $(\psi_m\colon \Omega\rightarrow \B^q)$ be a sequence of holomorphic injective maps with $\psi_m(z_m)=0$ and
$$
\psi_m(\Omega)\supset B(0,S_\Omega(z_m)).
$$

Given a compact subset $K\subset \B^q$, the map $\psi_m^{-1}$ is well defined on $K$ when $m$ is large enough. Since $f^{m}\circ\psi_{m}^{-1} (0)=z_0$ for all $m\ge 0$ and since $\Omega$ is taut, there exists a subsequence $(m_0(h))$ such that the sequence $(f^{m_0(h)}\circ\psi_{m_0(h)}^{-1} )$ converges uniformly on compact subsets to a holomorphic map $\alpha_0\colon \B^q\to\Omega$ and $\alpha_0(0)=z_0$. Similarly, there exists a subsequence  $(m_1(h))$ of $(m_0(h))$ such  that the sequence $(f^{m_1(h)-1}\circ\psi_{m_1(h)}^{-1} )$ converges uniformly on compact subsets to a holomorphic map $\alpha_1\colon \B^q\to\Omega$ and $\alpha_1(0)=z_1$.
Iterating this procedure we obtain a family of subsequences $\{(m_n(h))_{h\geq 0}\}_{n\geq 0}$ and a family of holomorphic maps
$$(\alpha_n\colon \B^q\to\Omega)_{n\geq 0}$$ 
such that  $$f^{m_n(h)-n}\circ\psi_{m_n(h)}^{-1}  \stackrel{h\to\infty}\longrightarrow \alpha_n$$ uniformly on compact subsets and $\alpha_n(0)=z_n$.
Notice that for all $m\geq n\geq 0$,
\begin{equation}
\label{backjonsnow}
\alpha_{n}=f^{m-n}\circ\alpha_m.
\end{equation}
Consider the diagonal sequence $\nu(h):=m_h(h)$ which for all $j\geq 0$ is eventually a subsequence of $(m_j(h))_{h\geq 0}.$

Consider the sequence $\beta_{\nu(h)}:=\psi_{\nu(h)}\circ\alpha_{\nu(h)}$. Notice that  $\beta_{\nu(h)}(0)=0$  for all $h\geq 0$. By the tautness of $\Omega$, up to extracting a further subsequence of $\nu(h)$ if necessary, we may assume that the sequence $(\beta_{\nu(h)})$ converges uniformly on compact subsets to a holomorphic map $\alpha\colon \B^q\to \B^q.$

\begin{lemma}
For all $j\geq 0$, we have
\begin{equation}\label{backstabilizza}
\alpha_j\circ\alpha=\alpha_j.
\end{equation}
\end{lemma}
\begin{proof}
Let $z\in \B^q$. For all positive integers $h$ such that $\nu (h)\geq j$,
$$\alpha_j(z)=f^{\nu(h)-j}\circ\alpha_{\nu(h)}(z)=f^{\nu(h)-j}\circ\psi_{\nu(h)}^{-1}\circ\psi_{\nu(h)}\circ\alpha_{\nu(h)}(z)\stackrel{h\to\infty}\longrightarrow \alpha_j\circ\alpha(z).$$
\end{proof}
\begin{lemma}
The map $\alpha\colon  \B^q\to \B^q$ is a holomorphic retraction, that is
$$\alpha\circ \alpha=\alpha.$$
\end{lemma}
\begin{proof}
Let $z\in \B^q$.
From \eqref{backstabilizza} we get, for all $h\geq 0$ big enough,
$$\beta_{\nu(h)}\circ\alpha(z)=\psi_{\nu(h)}\circ \alpha_{\nu(h)}\circ\alpha(z)=\psi_{\nu(h)}\circ\alpha_{\nu(h)}(z)=\beta_{\nu(h)}(z),$$
and the result follows since $\beta_{\nu(h)}\to \alpha$.
\end{proof}
Define $Z:=\alpha(\B^q)$. Being a holomorphic retract, it is a closed complex submanifold of $\B^q$, biholomorphic to a $k$-dimensional ball $\B^k$, with $0\leq k\leq q$. 

By the universal property of the inverse limit $(\Theta, V_n)$, there exists a unique map $\Psi\colon Z\to \Theta$ such that 
$$
\alpha_n=V_n\circ\Psi, \quad \forall\, n\geq 0.
$$
The mapping $\Psi$ sends the point $z\in Z$ to the backward orbit $(\alpha_m(z))_{m\geq 0}$.

\begin{lemma}\label{backgrosso}
The map $\Psi\colon  Z\to\Theta$ is injective and $\Psi(Z)=[z_m]$.
\end{lemma}
\begin{proof}
We first prove injectivity. Let $z,w\in Z$ such that $\alpha_m(z)=\alpha_m(w)$ for all $m\ge 0$. It follows that 
$$
\alpha(z)=\lim_{h\to\infty}\psi_{\nu(h)}\circ\alpha_{\nu(h)}(z)=\lim_{h\to\infty}\psi_{\nu(h)}\circ\alpha_{\nu(h)}(w)=\alpha(w).
$$
Since $\alpha|_Z$ is the identity, we conclude that $z=w$ and that $\Psi$ is injective.

Given $z\in Z$ it follows that
$$
\sup_m k_\Omega(\alpha_m(z),z_m)=\sup_m k_\Omega(\alpha_m(z),\alpha_m(0))\le k_Z(z,0)<\infty,
$$
proving that $\Psi(Z)\subset [z_m]$. On the other hand, given $(w_m)\in[z_m]$ we have
$$
\sup_m k_{\B^q}(\psi_m(w_m),0)=\sup_m k_{\B^q}(\psi_m(w_m),\psi_m(z_m))<\infty,
$$
By taking a subsequence of $\nu(h)$ if necessary, we may therefore assume that the sequence $\psi_{\nu(h)}(w_{\nu(h)})$ converges to a point $w\in \B^q$. Now we notice that for all $m\ge 0$,
$$
\alpha_m(w)=\lim_{h\to\infty}f^{\nu(h)-m}\circ\psi_{\nu(h)}^{-1}\circ\psi_{\nu(h)}(w_{\nu(h)})=\lim_{h\to\infty}f^{\nu(h)-m}(w_{\nu(h)})=w_m.
$$
 Now notice that 
$$
\alpha(w)=\lim_{h\to\infty}\psi_{\nu(h)}\circ\alpha_{\nu(h)}(w)=\lim_{h\to\infty}\psi_{\nu(h)}(w_{\nu(h)})=w,
$$
proving that $w\in Z$, and therefore that $\Psi(Z)= [z_m]$.
 \end{proof}
 
\begin{lemma}\label{backx4}
For all $n\geq 0$,
\begin{equation*}
\lim_{m\to\infty}  \alpha_m^*\,(f^{n})^*\, k_\Omega= k_Z.
\end{equation*}
\end{lemma}
\begin{proof}
First of all we notice that for every $m\ge 0$ and $x,y\in Z$
$$
k_\Omega(\alpha_m(x),\alpha_m(y))=k_\Omega((f\circ\alpha_{m+1})(x),(f\circ\alpha_{m+1})(y))\le k_\Omega(\alpha_{m+1}(x),\alpha_{m+1}(y)),
$$
therefore the limit for $m\to\infty$ is well defined,

For all $n\ge 0$ we have that
$$
 \lim_{m\to\infty}k_\Omega(f^n \circ \alpha_m(x),f^n \circ \alpha_m(y))= \lim_{m\to\infty}k_\Omega(\alpha_{m-n}(x),\alpha_{m-n} (y)),
$$
proving that 
$$
\lim_{m\to\infty}  \alpha_m^*\,(f^{n})^*\, k_\Omega=\lim_{m\to\infty}  \alpha_m^*\, k_\Omega.
$$

By non-expansiveness of the Kobayashi distance we have
$$
\lim_{m\to\infty}k_\Omega(\alpha_m(x),\alpha_m(y))\le k_Z(x,y).
$$
To obtain the inverse inequality denote $z_h=\psi_{\nu(h)}\circ\alpha_{\nu(h)}(x)$ and $w_h=\psi_{\nu(h)}\circ\alpha_{\nu(h)}(y)$. Then $z_h$ converges to $x$ and $w_h$ converges to $y$. Fix $\varepsilon>0$, then there exists a ball $B=B(0,r)\subset\subset \B^q$, with radius close enough to $1$ that contains both $x,y$, and such that for some $h_0\ge 0$, we have
$$
k_B(z_h,w_h)\le k_{\B^q}(x,y)+\varepsilon,\qquad\forall h\ge h_0.
$$
Let $h_1\ge 0$ such that for all $h\ge h_1$ we have $\psi_{\nu(h)}(\Omega)\supset B$. Then for all $h\ge\max\{h_0,h_1\}$ we have
$$
k_\Omega(\alpha_{\nu(h)}(x),\alpha_{\nu(h)}(y))\le k_B(z_h,w_h)\le k_{\B^q}(x,y)+\varepsilon,
$$
proving that for every $\varepsilon>0$
$$
\lim_{m\to\infty}k_\Omega(\alpha_{m}(x),\alpha_{m}(y))\le k_Z(x,y)+\varepsilon,
$$
where we used that fact that $x,y\in Z$ and $k_{\B^q}|_Z=k_Z$.
\end{proof}

We are now ready to prove  Proposition \ref{backcore}. Points $(a)$, $(b)$, and $(c)$ correspond precisely to \eqref{backjonsnow}, Lemma \ref{backgrosso} and Lemma \ref{backx4}.

It remains to prove the Universal property $(d)$. Let $(\beta_n\colon Q\to \Omega)$ be a family of holomorphic maps satisfying  $\beta_n=f^{m-n}\circ \beta_m$ for all  $m\geq n\geq 0$ and $(\beta_m(x))\in[z_m]$ for every $x\in Q$. By the universal property of the inverse limit, there exists a unique map $\Phi\colon Q\to \Theta$ such that $\beta_n=V_n\circ \Phi$ for all $n\geq 0$. It is not hard to show that $\Phi(x)=(\beta_m(x))\in [z_m]$.

By Lemma \ref{backgrosso} the map $\Psi\colon Z\to [z_m]$ is a bijection. Therefore we may set
$$
\Gamma\coloneqq \Psi^{-1}\circ\Phi \colon Q\to Z.
$$
Given $x\in Q$ it follows that 
$$
(\beta_m(x))=(\Psi\circ\Gamma)(x)=\Phi(x)=((\alpha_m\circ\Gamma)(x)),
$$
proving that $\beta_m=\alpha_m\circ\Gamma$ for all $m\ge 0$. Given another $\Gamma'$ with the same properties it is immediate to show that $\Gamma'=\Psi^{-1}\circ\Phi$, proving uniqueness of the map $\Gamma$.

Finally the map $\Gamma$ is holomorphic. Indeed given $x\in Q$, since $(\beta_m(x))\in[z_m]$ it follows that
$$
\sup_m k_{\B^q}(\psi_m\circ\beta_m(x),0)= \sup_m k_{\B^q}(\psi_m\circ\beta_m(x),\psi_m(z_m))\le k_\Omega(\beta_m(x),z_m)<\infty.
$$ 
The domain $\Omega$ is taut and the sequence $\psi_m\circ\beta_m$ is not compactly divergent. By taking a subsequence of $\nu(h)$ if necessary, we may therefore assume that $\psi_{\nu(h)}\circ\beta_{\nu(h)}\to \beta$ where $\beta\colon Q\to \B^q$ is an holomorphic function. Finally for every $x\in Q$ we have that
$$
\Gamma(x)=\alpha\circ\Gamma(x)=\lim_{h\to\infty}\psi_{\nu(h)}\circ\alpha_{\nu(h)}\circ\Gamma(x)=\beta(x),
$$
proving that $\Gamma$ is holomorphic, which concludes the proof of point $(d)$ and of Proposition \ref{backcore}.

\section{Main result on strongly convex domains}
In this section we apply the results of the previous section to the case of strongly convex domains, and we prove Theorem~\ref{teoconvessoindietro}.
First of all,  on a strongly convex domain it is easy to characterize when a canonical pre-model is $0$-dimensional (and thus  its base space is a point $\{\star\}$).
Recall that a self-map $f$ is called {\sl strongly elliptic} if it is elliptic and its limit manifold is a fixed point $\{p\}$.

\begin{lemma}
Let $\Omega\subset \subset \C^q$ be a  strongly convex domain with $C^3$ boundary, and let $(Z,\ell,\tau)$ be a canonical pre-model associated with a class $\mathscr{C}$ of backward orbits. Then $Z$ is $0$-dimensional if and only if $f$ is strongly elliptic and the class  $\mathscr{C}$ contains only the constant orbit $p$, where $\{p\}$ is the limit manifold of $f$.
\end{lemma}

\begin{proof}
Assume $Z=\{\star\}$, and set $p:=\ell(\star)$. Clearly $p$ is fixed. By assumption the backward orbit $(\ell\circ\tau^{-m}(\star)))$, which is contantly equal to $p$, is in the class $\mathscr{C}$.

Let $\mathcal M$ be  the limit manifold where the forward dynamics of $f$ converges. 
The restriction  $f|_{\mathcal{M}}$ is an automorphism of $\mathcal{M}$ and $(\mathcal{M},{\sf id}, f|_{\mathcal{M}} )$ is a pre-model for $f$. 
Since $p\in \mathcal{M} $ is a fixed point for $f$, the pre-model $(\mathcal{M},{\sf id}, f|_{\mathcal{M}} )$ is associated with the class $\mathscr{C}$. By the universal property of the canonical pre-model $(Z,\ell,\tau)$, there exists a morphism $\hat\eta\colon(\mathcal{M},{\sf id}, f|_{\mathcal{M}} )\to (Z,\ell,\tau)$, which means that the identity map ${\sf id}\colon \mathcal{M}\to \mathcal{M}$ is equal to the constant map $  \mathcal{M}\to \{p\}$. Hence  $\mathcal M=\{p\}$ and the map $f$ is strongly elliptic.

%

 By \cite[Lemma 2.9]{AbRa} any other backward orbit with bounded step $(w_m)$ converges to $\partial \Omega$. Since $\Omega$ is complete hyperbolic, it follows that $k_\Omega(w_m,p)\to +\infty$, and thus $(w_m)\not\in \mathscr{C}.$
The converse is immediate.
\end{proof}

Let $\Omega\subset \subset \C^q$ be a bounded strongly convex domain with $C^3$ boundary.
Let $f\colon \Omega\to \Omega$ be a  holomorphic self-map, and let $\zeta\in \partial \Omega$ be  a repelling boundary point with dilation $\lambda_\zeta>1$.
\begin{defi}
The {\sl stable subset} $\mathcal{S}(\zeta)$ of $\zeta$ is the set of starting points of backward orbits with bounded step converging to $\zeta$.
We say that a pre-model  $(\Lambda,h,\varphi)$ is {\sl associated} with the boundary repelling point $\zeta$ if  for some (and hence for any) $x\in \Lambda$ we have $$\lim_{n\to \infty}h\circ\varphi^{-n}(x)= \zeta.$$
\end{defi}
We will later prove the two following results.
\begin{teo}[Uniqueness of backward orbits]\label{backuniq}
Let  $(x_m)$ and $(y_m)$ be two  backward orbits with bounded step, both converging to the boundary repelling fixed point $\zeta\in \partial \Omega$. Then 
$$
\lim_{m\to\infty}k_{\Omega}(x_m,y_m)< \infty.
$$
\end{teo}
\begin{teo}[Existence of  backward orbits]\label{backward}
 Assume further that $\partial \Omega$ is   $C^4$. Then there exists a backward orbit $(z_m)$ with step $\log\lambda_\zeta$ converging to $\zeta$.
\end{teo}

As a consequence, the family of backward orbits with bounded step converging to $\zeta$ is non-empty and  consists of a unique equivalence class $[z_m]$. This, together with Theorem  \ref{backastratto}  gives the following.

\begin{teo}\label{finalebackward}
Let $\Omega\subset \C^q$ be a bounded strongly convex domain with $C^4$ boundary.
Let $f\colon \Omega\to \Omega$ be a  holomorphic self-map, and let $\zeta\in \partial \Omega$ be  a  boundary repelling fixed point.   
Then there exist
\begin{enumerate}
\item an integer $k$ such that $1\leq k\leq q$, 
\item a hyperbolic automorphism $\tau\colon \HH^k\to \HH^k$ of the form
\begin{equation}\label{backlomo1}
\tau(z_1,z')= \left(\frac{1}{\lambda_\zeta} z_1,\frac{e^{it_1}}{\sqrt \lambda_\zeta}z'_1,\dots, \frac{e^{it_{k-1}}}{\sqrt \lambda_\zeta}z'_{k-1} \right),
\end{equation}
where $t_j\in \R$ for $1\leq j\leq k-1$, 
\item a holomorphic mapping
$\ell\colon  \HH^k\to \Omega$,
\end{enumerate}
 such that the triple
 $(\HH^k,\ell,\tau)$ is a   pre-model for $f$ associated with $\zeta$  satisfying the following universal property: if  $(\Lambda,h,\varphi)$ is a pre-model associated with $\zeta$, then 
 there exists a unique morphism $\hat\eta\colon (\Lambda,h,\varphi)\to (\HH^k,\ell,\tau)$.
 
Moreover, the image $\ell(\HH^k)$ is equal to the stable subset $\mathcal{S}(\zeta)$, and
$$K\hbox{-}\lim_{z\to\infty }\ell(z)=\zeta.$$
\end{teo}
\begin{proof}
By Theorems \ref{backward} and \ref{backuniq}, the family of backward orbits with bounded step converging to $\zeta$ is non-empty and  consists of a unique equivalence class $[z_m]$. Let $(\HH^k,\ell, \tau)$  be the canonical pre-model for $f$ associated with $[z_m]$. Notice that
 $$c(\tau)\leq s_1(z_m)=\log\lambda.$$
 To obtain the opposite inequality, notice that, if $p\in \Omega$,
 $$\log\lambda\leq \liminf_{m\to\infty}k_\Omega(p,z_{m+1})-k_\Omega(p,z_m).$$
Hence, if $n\geq 0$ is fixed, 
 $$n\log\lambda\leq \liminf_{m\to\infty}k_\Omega(p,z_{m+n})-k_\Omega(p,z_m)\leq \liminf_{m\to\infty}k_\Omega(z_{m+n},z_m)\leq s_n(z_m).$$ Thus $$c(\tau)=\inf_{n\in \N}\frac{s_n(z_m)}{n}\geq \log\lambda.$$
 Now we can change variables in $\HH^k$ to put $\tau$ in the form (\ref{backlomo1}).
 
We finally study the regularity at $\infty$ of the intertwining mapping $\ell$.
Consider the backward orbit with bounded step $((\lambda^ni,0))$ in $\HH^k$ for $\tau$.  
 Clearly  $((\lambda^ni,0))$ converges to $\infty$ and $(\ell(\lambda^ni,0))$ is a backward orbit for $f$ which converges to $\zeta\in\partial\Omega$. Then  \cite[Theorem 5.6]{ArBr}  yields $K\hbox{-}\lim_{z\to \infty}\ell(z)=\zeta$.

\end{proof}

\section{Uniqueness of backward orbits}
In this section we prove Theorem \ref{backuniq}. We remark that the $C^4$-smoothness of $\partial \Omega$ is only required in the proof of Theorem \ref{backward}), which is done in the next section. Here it will be sufficient to assume that the domain $\Omega$ has $C^3$ boundary.

Given a backward  orbit $(x_m)$ one can always assume that it is indexed by integers $m\in\mathbb Z$, by defining $x_{-m}:= f^{m}(x_0)$ for all $m\geq 0$. 
\begin{lemma}\label{trasl}
Let $\Omega\subset \C^q$ be a bounded strongly convex domain with $C^3$ boundary.
Let $f\colon \Omega\to \Omega$ be a  holomorphic self-map, and let $\zeta\in \partial \Omega$ be  a boundary repelling fixed point.   

Let $(x_m)$ and $(y_m)$ be two  backward orbits with bounded step, both converging to $\zeta$. Then $\lim_{m\to+\infty}k_{\Omega}(x_m,y_m)< \infty$ if and only if
\[
\lim_{n\to+\infty}\inf_{m\in\mathbb Z}k_{\Omega}(x_n,y_m)<\infty.
\]
\end{lemma}
\begin{proof}
The proof is the same as in \cite[Lemma 2]{ArGu}.
\end{proof}

We recall some definitions and classical results (see e.g. \cite{pappa}). 
 \begin{definition}
 A metric space $(X,d)$ is {\sl geodesic} if any two points are joined by a geodesic segment.
\end{definition}
\begin{teo}[Generalized Hopf-Rinow]
If an inner metric space $(X,d)$ is complete and locally compact,  it is  geodesic.
\end{teo}
Hence, a complete Kobayashi hyperbolic manifold is geodesic.
\begin{definition}
Let $\delta>0$.
 A geodesic metric space $(X,d)$  is {\sl $\delta$-Gromov hyperbolic} if   every geodesic triangle is {\sl $\delta$-slim}, meaning that every side is contained in a $\delta$-neighborhood of the union of the other two sides. A geodesic metric space $(X,d)$  is {\sl Gromov hyperbolic} if it is  $\delta$-Gromov hyperbolic for some $\delta>0$.
\end{definition}

\begin{teo}[{\cite[Theorem 1.4]{BaBo}}]
Let $\Omega\subset \mathbb C^q$ be a bounded strongly pseudoconvex domain with $C^2$ boundary and $k_\Omega$ be its Kobayashi distance. Then $(\Omega,k_\Omega)$ is Gromov hyperbolic.
\end{teo}

\begin{prop}\label{filippo}
Let $(X,d)$ be a   geodesic $\delta$-Gromov hyperbolic metric space.
If $\gamma$ is a geodesic line,  $x_0\in  \gamma$, $z\in X$, and if  $z_\gamma$ denotes a point in $ \gamma$ such that 
$d(z,z_\gamma)=d(z,\gamma)$, then 
$$  d(x_0,z)\geq  d(x_0,z_\gamma)+d(z_\gamma,z)-6\delta.$$         
\end{prop}
\begin{proof}
If $d(z_\gamma, z)\leq 3\delta$, the result follows from the triangular inequality.
Assume thus that $d(z_\gamma, z)> 3\delta$.  
Let $(x_0, z_\gamma)$ denote the portion of $\gamma$ between $x_0$ and $z_\gamma$.
Let $(z_\gamma, z)$ be a geodesic segment connecting $z_\gamma$ to $z$, and let 
$(x_0,z)$ be a geodesic segment connecting $x_0$ to $z$. Let $x$ be a point in $(z_\gamma, z)$ such that 
 $d(x,z_\gamma)=2\delta$. Since every geodesic triangle is $\delta$-slim, there exists a point $y$ in 
$(x_0, z_\gamma)\cup (x_0,z)$ such that $d(y,x)<\delta.$ From $d(x,\gamma)=2\delta$ it  follows that  $y\in (x_0, z)$, and from the triangle inequality we have $d(y,z_\gamma)<3\delta.$ Using twice more the triangular inequality we obtain
$$d(x_0,y)\geq d(x_0,z_\gamma)-d(y,z_\gamma),$$
$$d(y,z)\geq d(z,z_\gamma)-d(y,z_\gamma).$$
Summing the two inequalities yields the result.
\end{proof}

\begin{definition}
Let $\Omega\subset \C^q$ be a bounded strongly convex domain with $C^3$ boundary. Let $p\in \Omega$, $\zeta\in \partial \Omega$ and let $\gamma\subset \Omega$ be  the geodesic ray connecting $p$ to $\zeta$.  Given $M>1$ we denote $$A(\gamma,M):=\left\{ z\in \Omega\colon  k_\Omega(z,\gamma)<{\log M}\right\}.$$
\end{definition}

We now show that the Koranyi regions are comparable to the regions $A(\gamma,M)$. Let $\delta>0$ be such that $(\Omega,k_\Omega)$ is $\delta$-Gromov hyperbolic.
\begin{lemma}\label{gromov}
Let $\Omega\subset \C^q$ be a bounded strongly convex domain with $C^3$ boundary. Let $\zeta\in \partial \Omega$, $p\in \Omega$ and let $\gamma\subset \Omega$ be  the geodesic ray connecting $p$ to $\zeta$.
Then  for every $M>1$, 
 $$A(\gamma,M)\subset K_\Omega(p,\zeta,M)\subset A(\gamma,Me^{6\delta}).$$
\end{lemma}
\begin{proof}
Let $z\in A(\gamma,M)$, and let $y\in \gamma$ be a point such that $k_{\Omega}(z,y)< \log M$.
Let $w\in \gamma$ close enough to $\zeta$. Then
$$ k_{\Omega}(p,z)+k_{\Omega}(z,w)-k_{\Omega}(p,w)\leq k_{\Omega} (p,y)+ k_{\Omega}(y,z)+ k_{\Omega}(z,y)+ k_{\Omega}(y,w)-k_{\Omega}(p,w)=2\,k(y,z).$$
Letting $w$ go to $\zeta$ on the geodesic ray $\gamma$, we obtain that $$\log h_{\zeta,p}(z)+k_\Omega(z,p)\leq 2\,k(y,z)< 2\log M,$$ and therefore that $z\in K_\Omega(p,\zeta,M)$.

Conversely, let $z\in K_\Omega(p,\zeta,M)$. Denote by $\tilde \gamma$ the geodesic line containing $\gamma$ and  let $y\in \tilde\gamma$ be the closest point to $z$. Let $w\in \gamma$.
Applying  Proposition  \ref{filippo} twice we obtain
\begin{equation}\label{rivia}
k_{\Omega}(p,z)+ k_{\Omega}(z,w)-k_{\Omega}(p,w)\geq k_{\Omega}(p,y)+k_{\Omega}(y,w)+2k_{\Omega}(y,z)-12\delta-k_{\Omega}(p,w).
\end{equation}
We now have two cases. If $y\in\gamma$, from \eqref{rivia} we get    
$$k_{\Omega}(p,z)+ k_{\Omega}(z,w)-k_{\Omega}(p,w)\geq2k_{\Omega}(y,z)-12\delta.$$
If  $y\not\in \gamma$, then  $ k_{\Omega}(y,w)-k_{\Omega}(p,w)=k (y,p)$, and thus from \eqref{rivia} we get
$$k_{\Omega}(p,z)+ k_{\Omega}(z,w)-k_{\Omega}(p,w)\geq 2k_{\Omega}(p,y)+2k_{\Omega}(y,z)-12\delta\geq 2k_\Omega(z,p)-12\delta.$$

 In both cases, letting $w$ go to $\zeta$ on the geodesic ray $\gamma$, we obtain that $$k_{\Omega}(z,\gamma)<\log M +6\delta,$$ and thus that $z\in A(\gamma,Me^{6\delta})$.
\end{proof}

\begin{lemma}\label{entrakor}
Let $\Omega\subset \C^q$ be a bounded strongly convex domain with $C^3$ boundary.
Let $f\colon \Omega\to \Omega$ be a  holomorphic self-map, and let $\zeta$ be  a boundary repelling fixed point.
Let $(z_m)$ be a backward orbit with bounded step converging to $\zeta$. Then for every $p\in\Omega$ there exists $M>1$ so that
$$
z_m\in K_\Omega(p,\zeta,M),\qquad\forall m\ge 0.
$$
\end{lemma}
\begin{proof}
Let $p\in \Omega$.
By the definition of the dilation $\lambda_\zeta$ we have
$$
\liminf_{n\to\infty}\left( k_{\Omega}(p,z_{n+1})- k_{\Omega}(p,z_{n})\right)\ge \log \lambda_\zeta.
$$
For all $n\geq 0$ define $s_n$ by $-\log s_n:=k_\Omega(p,z_n).$
It follows that there exist $ \lambda_\zeta^{-1}<c<1$ and $n_0\geq 0$ such that 
$s_{n+1}\leq c s_n$ for all $n\geq n_0$. Up to shifting the sequence $(z_n)$ we may thus assume that 
$$s_{n+k}\leq c^{k}s_n, \quad \forall n\geq 0, k\geq 0.$$
The proof now follows as in \cite[Errata Corrige--Lemma 2.16]{AbRa}.

\end{proof}

We are now ready to prove Theorem \ref{backuniq}.
\begin{proof}[Proof of Theorem \ref{backuniq}]
Let  $(x_m)_{m\in \mathbb{Z}}$ and $(y_m)_{m\in \mathbb{Z}}$ be two  backward orbits with bounded step, both converging to the boundary repelling fixed point $\zeta\in \partial \Omega$. Let $p\in\Omega$ and $\gamma$ be the geodesic ray starting at $p$ and ending in $\zeta$. By Lemma \ref{entrakor} the sequence  $(x_m)_{m\in\mathbb{N}}$ is contained in a Koranyi region $K_\Omega(p,\zeta, M)$ for some $M>1$, and thus by Lemma \ref{gromov} it is contained in the region $A(\gamma,Me^{6\delta})$. We claim that there exists $R>0$ such that 
$$A(\gamma,Me^{6\delta})\subset \left\{ z\in \Omega \colon \inf_{m\in\mathbb{Z}} k_\Omega(z,y_m)<R\right\}.$$
Once the claim is proved, the result follows by Lemma \ref{trasl}.

It is enough to show that there exists a constant $R'>0$ such that for all $w\in \gamma$, we have $\inf_{m\in \mathbb{Z}} k_{\Omega}(w,y_m)<R'$.
Since  $(y_m)_{m\in \mathbb{N}}$  is also contained in a Koranyi region $K_\Omega(p,\zeta,M')$, if we write $C=6\delta+\log M'$ it follows that
$k_{\Omega}(y_m, \gamma)< C$ for all $m\in \mathbb{N}$.
Let $a_m$ be a point in $\gamma$ such that $k_{\Omega}(y_m,a_m)<C$. Clearly $a_m\to \zeta$. Let $w$ be a point in the portion of $\gamma$ which connects $a_0$ to $\zeta$. Then there exists $m(w)$ such that $w$ belongs to the portion of $\gamma$ which connects $a_{m(w)}$ to $a_{m(w)+1}$.
Hence $$\inf_m k_{\Omega}(w, y_m)\leq C+k_{\Omega}(a_{m(w)},a_{m(w)+1})\leq C+2C+ k_{\Omega}(y_{m(w)},y_{m(w)+1})\leq 3C+ \sigma(y_m).$$
Letting $R'=3C+\sigma(y_m)$ we obtain that $\inf_{m\in \mathbb{Z}} k_{\Omega}(w,y_m)<R'$ holds for every $w\in\gamma$ sufficiently close to $\zeta$. By taking a bigger $R'$ if necessary, we may therefore assume that the inequality holds for all $w\in\gamma$, concluding the proof of the theorem.
\end{proof}

\section{Existence of backward orbits}
In this section we prove Theorem \ref{backward}. In the case of the ball $\B^q$ this result was  first proved  in  \cite{O}  assuming that the boundary repelling fixed point is isolated, and then for a general boundary repelling fixed point in \cite{ArGu}. 
For strongly convex domains with $C^3$ boundary, it was proved in \cite{AbRa} in the case of an isolated boundary repelling fixed point.

\subsection{Preparatory results in the ball}

In the following result we reformulate the crucial part of the proof of \cite[Theorem 2]{ArGu} as a purely geometric statement.
\begin{prop}\label{limits}
Let $(x_n),(y_n)\in \B^q$ be two sequences satisfying
\begin{equation}\label{allineano}
\lim_{n\to\infty}[k_{\B^q}(0,x_n)-k_{\B^q}(0,y_n)]=\lim_{n\to\infty}k_{\B^q}(x_n,y_n)=L>0.
\end{equation}
Suppose further that there exists $R>0$ and $\zeta\in \partial \B^q$ so that $x_n\in E_{\B^q}(0,\zeta,R)$ and $y_n\not\in E_{\B^q}(0,\zeta,R)$ for every $n\in\mathbb N$. Then $(x_n)$ and $(y_n)$ are relatively compact in $\B^q$.
\end{prop}
\begin{proof}
Assume that this were not the case. Since $x_n\in E_{\B^q}(0,\zeta,R)$ and since the Kobayashi distance between $x_n$ and $y_n$ is bounded, by taking a subsequence of $x_n$ if necessary, we may then assume that $x_n,y_n\to\zeta$.

The automorphism group of the unit ball is transitive. Therefore for every positive integer $n$ we may find  $\sigma_n\in {\rm Aut}({\mathbb B}^q)$ so that $\sigma_n(x_n)=0$. By composing such map with a rotation, we may further suppose that $\sigma_n(\zeta)=e_1$. 

By \eqref{differentpole} and by invariance of the Kobayashi distance under automorphisms, we obtain that
\begin{align}
\label{relazioniorosfere}
h_{e_1,\sigma_n(0)}(z)=h_{e_1,0}(z)h_{e_1,\sigma_n(0)}(0)=h_{e_1,0}(z)h_{\zeta,0}(x_n),
\end{align}
and therefore that 
\begin{equation}
\label{ultimo}
E_{\B^q}(0,e_1,1)\subset E_{\B^q}(\sigma_n(0),e_1,R)=\sigma_n(E_{\B^q}(0,\zeta,R)).
\end{equation}

We claim that $\sigma_n(0)\to e_1$.
Notice that for every $n\in\mathbb N$ large enough we have $h_{\zeta,0}(x_n)> Re^{-2L}$. Indeed if for some $n$ sufficiently large this were not the case, then we would have 
$$
 h_{\zeta,0}(y_n)\le e^{k_{\B^q}(x_n,y_n)}h_{\zeta,0}(x_n)< R,
$$
contradicting the fact that $y_n\not\in E_{\B^q}(0,\zeta,R)$. Letting $z=\sigma_n(0)$ in \eqref{relazioniorosfere} we conclude that  for $n$ large enough $$h_{e_1,0}(\sigma_n(0))=h_{\zeta,0}(x_n)^{-1}<e^{2L}/R,$$ showing that the sequence $\sigma_n(0)$ is  eventually contained in $ E_{\B^q}(0,e_1, e^{2L}/R)$. The point $x_n$ converges to $\zeta$, and therefore $k_{\B^q}(0,x_n)\to\infty$. By invariance of the Kobayashi distance it follows that $k_{\B^q}(0,\sigma_n(0))\to\infty$, and thus that $\sigma_n(0)\to e_1$, proving the claim.

Let $0<\alpha<1$ and define $z_n\in\B^q$ as
$$
z_n:=-\alpha\frac{\sigma_n(0)}{\Vert\sigma_n(0)\Vert}.
$$
If we write $\beta:=\log\frac{1+\alpha}{1-\alpha}$, then for every positive integer $n$ we have $k_\B(0,z_n)=\beta$. Since $\sigma_n(0)\to e_1$ it follows that $z_n\to z_\infty=(-\alpha, 0,\dots ,0)$.

By \eqref{allineano} the sequence $\sigma_n(y_n)$ is relatively compact in $\mathbb B^q$. Therefore by taking a subsequence if necessary, we may assume that $\sigma_n(y_n)\to y_\infty\in \mathbb{B}^q$. Notice that since $y_n\not\in E_{\B^q}(0,\zeta,R)$, then by \eqref{ultimo} we must have $\sigma_n(y_n)\not\in E_{\B^q}(0,e_1,1)$, thus that $y_\infty\not\in E_{\B^q}(0,e_1,1)$.

We claim that $k_{\mathbb{B}^q}(z_\infty,y_\infty)<L+\beta$. Indeed, since $k_{\mathbb{B}^q}(0,y_\infty)=L$ and $k_\B(0,z_\infty)=\beta$, we get by triangular inequality that $k_{\mathbb{B}^q}(z_\infty,y_\infty)\leq L+\beta$. Equality holds if and only if $y_\infty$ is contained in the geodesic ray connecting the origin to $e_1$. But this is not possible since such geodesic is contained in $ E_{\B^q}(0,e_1,1)$.
Let thus $\delta>0$ be such that  $k_{\mathbb{B}^q}(z_\infty,y_\infty)< L +\beta-2\delta.$

By the last inequality and by \eqref{allineano} we may choose $n$ big enough such that $k_{\mathbb{B}^q}(z_n,\sigma_n(y_n))<L+\beta -\delta$ and  $$ k_{\mathbb{B}^q}(\sigma_n(0),0)-k_{\mathbb{B}^q}(\sigma_n(0),\sigma_n(y_n))=k_{\mathbb{B}^q}(0,x_n)-k_{\mathbb{B}^q}(0,y_n)\ge L-\delta.$$ We conclude that
\begin{align*}
k_{\mathbb{B}^q}(\sigma_n(0),z_n)-k_{\mathbb{B}^q}(\sigma_n(0),\sigma_n(y_n))&=k_{\mathbb{B}^q}(\sigma_n(0),0)+k_{\mathbb{B}^q}(0,z_n)-k_{\mathbb{B}^q}(\sigma_n(0),\sigma_n(y_n))\\
&\geq \beta +L-\delta\\
&>k_{\mathbb{B}^q}(z_n,\sigma_n(y_n)),
\end{align*}
contradicting the triangular inequality. 
\end{proof}

Before starting the proof of Theorem \ref{backward}, we need to estimate  the Kobayashi distance of horospheres near the center $\zeta$.
\begin{lemma}
\label{eqregion}
Let $\zeta\in \partial\B^q$ and $R>0$. Write $k_E$ for the Kobayashi distance of the horosphere $E_{\B^q}(0,\zeta,R)$. Then for all $\varepsilon>0$  there exists $0<R_\varepsilon<R$ such that 
$$k_E(x,y)\leq k_{\B^q}(x,y)+\varepsilon,\qquad\forall x,y\in E_{\B^q}(0,\zeta,R_\epsilon).$$
\end{lemma}
\begin{proof}
Consider the change of coordinates from the unit ball to the Siegel half-space given by \eqref{cayley}. The horosphere $E_{\B^q}(0,\zeta,R)$ is mapped by such biholomorphism to the horosphere
$$E_{\HH^q}(I,\infty,R)=\left\{(z_1,z')\in \HH^q\colon {\rm Im}\,z_1>\|z'\|^2+\frac{1}{R}\right\}.$$
Consider the biholomorphism $T\colon E_{\HH^q}(I,\infty,R)\to  \HH^q$ given by $T(z):= ( z_1-i/R,z').$ For all $x,y\in E_{\HH^q}(I,\infty,R)$ we have that 
$k_E(x,y)=k_{\HH^q}(T(x),T(y))$.

Let $S>1/R>0$ be such that $k_{\HH}(\zeta,\zeta-i/R)\leq \varepsilon/2$ when ${\rm Im}\,\zeta>S$. Set $R_\varepsilon=1/S$.
 Given $x,y\in E_{\HH^q}(I,\infty,R_\varepsilon)$ we obtain that $$k_E(x,y)=k_{\HH^q}(T(x),T(y))\leq k_{\HH^q}(T(x),x)+k_{\HH^q}(x,y)+k_{\HH^q}(y,T(y)),$$
hence the result follows if we show that
$$
k_{\HH^q}(T(z),z)\le\varepsilon/2,\qquad \forall z\in E_{\HH^q}(I,\infty,R_\varepsilon).
$$
Let thus $z=(z_1,z')\in  E_{\HH^q}(I,\infty,R_\varepsilon)$. Consider the complex geodesic $i_{z'}:\HH\rightarrow \HH^q$ given by $i_{z'}(\xi)=(\xi+i\Vert z'\Vert^2,z')$.
Set $\zeta:=z_1-i\Vert z'\Vert^2$. Then $\zeta$ satisfies ${\rm Im}\,\zeta>S$ and $i_{z'}(\zeta)=z$, and  $i_{z'}(\zeta-i/R)=T(z)$.
 Therefore we have
$$
k_{\HH^q}(z,T(z))= k_{\HH}(\zeta,\zeta-i/R)\le\varepsilon/2.
$$

\end{proof}

\subsection{Localization of the Kobayashi distance near the boundary}

In this subsection we show that, up to changing  coordinates, we can compare $k_{\B^q}$ and  $k_\Omega$ in little $\B^q$-horospheres centered at $\zeta$, as the following result shows.

\begin{prop}\label{changeofcoordinates}
Let $\Omega$ be a bounded strongly convex domain with $C^4$ boundary and let $\zeta\in\partial\Omega$. Then there exists a change of coordinates in ${\rm Aut}(\mathbb C^q)$ so that in the new coordinates $\zeta=e_1$ and the following holds: for every $\varepsilon>0$ we may find $R_{\varepsilon}>0$ so that $E_{\B^q}(0,e_1,R_\varepsilon)\subset \Omega$ and
$$
k_{\B^q}(z,w)-\varepsilon\le k_\Omega(z,w)\le k_{\B^q}(z,w)+\varepsilon,\qquad\forall z,w\in E_{\B^q}(0,e_1,R_\varepsilon). 
$$
\end{prop}
We first need to  prove some preparatory results. The following is proved in  \cite[Lemma 2]{F} (see also \cite[Proposition 9.7.7]{Kr}).
\begin{prop}
\label{feffer}
Let $\Omega\subset \C^q$ be a domain and let $\zeta\in \partial \Omega$ be a $C^4$-smooth strongly pseudoconvex point. There exists  a biholomorphic mapping $ w$ defined on a neighborhood $V$ of $\zeta$ sending $\zeta$ to  the origin, and sending $\Omega\cap V$ to a region with local   defining function at the origin of the form
\begin{equation}
\label{cheoleupagher}
\psi(w)=-\textrm{Im}\,w_1 +\Vert w'\Vert^2- P_4(\textrm{Re}\, w_1,w',\overline{w'})+o(|\textrm{Re}\, w_1|^4+\Vert w'\Vert^4),
\end{equation}
with $(w_1,w')\in w(V)$. Here $P_4$ is a real-valued 4th-degree homogeneous polynomial in $\textrm{Re}\,w_1,w'$ and $\overline{w'}$ satisfying $P_4(\textrm{Re}\,w_1,w',\overline{w'})\ge C(|\textrm{Re}\, w_1|^4+\Vert w'\Vert^4)$, for some $C>0$. 
\end{prop}

After applying the (local) change of variables $w=w(z)$, the boundaries of $\Omega$ and of the Siegel half-space $\HH^q$ have a $3$-th order contact at the origin.

Consider the biholomorphism $\tilde {\mathcal  C}:\B^q\rightarrow \HH^q$  given by
\begin{equation}
\label{Cayley2}
\tilde {\mathcal  C}(z_1,z')=\left(i\frac{1-z_1}{1+z_1},\frac{z'}{1+z_1}\right),
\end{equation}
which is equal to the Cayley transform $\mathcal  C$ defined by \eqref{cayley} precomposed with the automorphism $(z_1,z')\mapsto (-z_1,z')$ of $\B^q$, and as such it maps $e_1$ to $0$ and $-e_1$ to infinity.

The horospheres of $\HH^q$ with pole $I=(i,0,\dots,0)$ and center the origin are the images under the map $\tilde{\mathcal C}$ of the horospheres of the ball with pole $0$ and center $e_1$. Therefore for every $R>0$ we can write them as
$$
E_{\HH^q}(I,0,R)=\left\{w\in  \mathbb C^q\,|\,\textrm{Im}\, w_1>\Vert w'\Vert^2 +\frac{|w_1|^2}{R}\right\}.
$$

Choose a constant $D>0$ so that, whenever $w$ is sufficiently close to the origin, we have

\begin{equation}
\label{comparazione}
C(|\textrm{Re}\, w_1|^4+\Vert w'\Vert^4)\le P_4(\textrm{Re}\, w_1,w',\overline{w'})\le \frac{D}{2}(|\textrm{Re}\, w_1|^4+\Vert w'\Vert^4).
\end{equation}

Fix $R>0$ and let  $t:=\frac{1}{R}$. The holomorphic map defined by $\Phi(w_1,z'):= (w_1+it,w')$ is an automorphism of $\C\mathbb{P}^q$ which fixes the line at infinity and  satisfies  $\Phi(\HH^q)=E_{\HH^q}(I,\infty,R)$.
Conjugating $\Phi$ with the involution
$(w_1,w')\mapsto (-\frac{1}{w_1}, -\frac{iw'}{w_1})$ we obtain the automorphism of $\C\mathbb{P}^q$
$$
T(w_1,w'):=\left(\frac{R w_1}{R-iw_1},\frac{R w'}{R-iw_1}\right),
$$
which fixes the line $w_1=0$ and satisfies  $T(\HH^q)= E_{\HH^q}(I,0,R)$.

Fix $0<R<D^{-1}$. Let $V$ and $w(z)$ as in the previous proposition and, up to taking a smaller $V$ if necessary, define the local biholomorphism $\eta\colon V\rightarrow \mathbb C^q$ as $\eta=T\circ w$. It is not hard to show that in the coordinates $\eta=\eta(z)$ a local defining function of $\Omega\cap V$ at the origin can be written as \begin{equation}
\label{cheoleupagherR}
\psi(\eta)=-\textrm{Im}\,\eta_1 + \Vert \eta'\Vert^2+\frac{|\eta_1|^2}{R}- P_4(\textrm{Re}\, \eta_1,\eta',\overline{\eta'})+o(|\textrm{Re}\, \eta_1|^4+\Vert \eta'\Vert^4),
\end{equation} 
where $P_4$ is the same polynomial as in Proposition \ref{feffer}.  Notice that  the higher order terms still do not depend on $\textrm{Im}\,\eta_1$.

\begin{lemma}
There exists $\rho>0$ so that,
$$
E_{\HH^q}(I,0,R)\cap B(0,\rho)\subset \eta(\Omega\cap V)\cap B(0,\rho)\subset \HH^q\cap B(0,\rho)
$$
\end{lemma}
\begin{proof}
The local defining function of $\Omega':=\eta(\Omega\cap V)$ at the origin is of the form
$$
\psi(\eta)=-\textrm{Im}\,\eta_1 + \Vert \eta'\Vert^2+\frac{|\eta_1|^2}{R}- r(\textrm{Re}\, \eta_1,\eta',\overline{\eta'}),
$$
where $r$ is the sum of $P_4$ and the  higher order terms. By \eqref{comparazione} it follows that whenever $\eta$ is sufficiently close to the origin, we have that
$$
0\le r(\textrm{Re}\, \eta_1,\eta',\overline{\eta'})\le D(|\textrm{Re}\, \eta_1|^4+\Vert \eta'\Vert^4).
$$
Suppose first that $\eta\in E_{\HH^q}(I,0,R)$ is close to the origin. Since $r$ is non negative it follows immediately that $\psi(\eta)<0$, proving that $\eta\in \Omega'$.

If on the other hand we have that if $\eta\in \Omega'$ is close to the origin, then 
\begin{align*}
0&<\textrm{Im}\, \eta_1-\frac{|\eta_1|^2}{R}-\Vert \eta'\Vert^2+r(\textrm{Re}\,\eta_1,\eta',\overline{\eta'})\\
&<\textrm{Im}\, \eta_1-\frac{|\eta_1|^2}{R}-\Vert \eta'\Vert^2+D(|\textrm{Re}\, \eta_1|^4+\Vert \eta'\Vert^4),
\end{align*}
and therefore
$$
D\Vert \eta'\Vert^4-\Vert \eta'\Vert^2+\textrm{Im}\,\eta_1-R^{-1}|\eta_1|^2+D|\textrm{Re}\,\eta_1|^4>0.
$$
As $\eta$ converges to the origin the corresponding second degree equation in $\Vert \eta'\Vert^2$ has two solutions $0<t_1<t_2$. The solution $t_2$ converges to $1/D$, while $$t_1=\frac{1-\sqrt{1-4D\textrm{Im}\,\eta_1+4DR^{-1}|\eta_1|^2-4D^2|\textrm{Re}\,\eta_1|^4}}{2D}\to 0 . $$
Hence if $\eta$ is small enough,   $\Vert \eta'\Vert^2<t_1$, which immediately implies  $\eta_1\neq0$.
Moreover, 
$$
\Vert \eta'\Vert^2<\textrm{Im}\, \eta_1+(D-R^{-1})|\eta_1|^2+O(|\eta_1|^3),
$$
proving that whenever $\eta\in \Omega'$ is sufficiently close to the origin, we have $\eta\in \HH^q$.
\end{proof}

Consider now the  biholomorphism $\tilde\eta:=\tilde{\mathcal C}^{-1}\circ\eta$ sending  $\zeta$ to $e_1$. 
 If $\varphi$ is a biholomorphism defined in a neighborhood of $\zeta$ such that 
$$\varphi(z)-\tilde \eta(z)=O(\Vert z-\zeta\Vert^{d+1}),$$
with $d$ sufficiently large, then the expression \eqref{cheoleupagherR} remains unchanged when we replace $\eta$ with $\tilde{\mathcal C}\circ \varphi$.

The proof of the previous lemma relies uniquely on the form of the boundary defining function \eqref{cheoleupagherR}. Therefore, up to taking a smaller $V$ if necessary, the lemma remains valid when we consider $\tilde{\mathcal C}\circ \varphi$ instead of $\eta$. We conclude that, for every given map $\varphi$ as above, there exists $\rho>0$ sufficiently small such that
$$
E_{\B^q}(0,e_1,R)\cap B(0,\rho)\subset \varphi(\Omega\cap V)\cap B(0,\rho)\subset \B^q\cap B(0,\rho).
$$
By jet interpolation in Anders\'en-Lempert theory \cite[Proposition 6.3]{AL} we may choose $\varphi$ to be an automorphism of $\mathbb C^q$. Since we can always find $0<R'<R$ so that $E_{\B^q}(0,e_1,R')\subset E_{\B^q}(0,e_1,R)\cap B(0,\rho)$, we conclude the following
\begin{lemma}
\label{coordinates}
Let $\Omega\subset \C^q$ be a domain and let $\zeta\in \partial \Omega$ be a $C^4$-smooth strongly pseudoconvex point.
Then there exists a change of coordinates in ${\rm Aut}(\mathbb C^q)$, and constants $\rho,R>0$ so that in the new coordinates we have $\zeta=e_1$ and the two following inclusions hold
$$
E_{\B^q}(0,e_1,R)\subset \Omega\qquad\text{and}\qquad  \Omega\cap B(e_1,\rho)\subset \B^q.
$$
\end{lemma}

Let $\Omega\subset \C^q$ be a bounded strongly convex domain with $C^3$ boundary.
Let $\mathscr F$ be the family of complex geodesics $\varphi:\D\rightarrow \Omega$ that satisfy (see \cite{CHL,H}) 
$$
d(\varphi(0),\partial D)=\max_{\tau\in\D}d(\varphi(\tau),\partial \Omega).
$$
Then by \cite[Proposition 1]{H} it follows that there exists  $C>0$ so that for every $\varphi\in\mathscr F$ and $\tau_1,\tau_2\in\D$ we have
\begin{equation}
\label{huang}
\Vert \varphi^{(j)}(\tau_1)-\varphi^{(j)}(\tau_2)\Vert\le C|\tau_1-\tau_2|^{1/4}, \quad j=0,1.
\end{equation}
{}
\begin{lemma}
\label{puissance}
Let $\Omega\subset \C^q$ be a bounded strongly convex domain with $C^3$ boundary and let $\zeta\in\partial\Omega$. Then given $\varepsilon>0$ we may find $\delta>0$ so that for every $z,w\in \Omega\cap B(\zeta,\delta)$ the geodesic segment  for the Kobayashi distance $\gamma$ connecting $z$ and $w$ is contained in $\Omega\cap B(\zeta,\varepsilon)$ and has euclidean length $\ell(\gamma)<\varepsilon$.
\end{lemma}
\begin{proof}
The fact that $\gamma$ is contained in $\Omega\cap B(\zeta,\varepsilon)$ for $\delta$ sufficiently small is an immediate consequence of \cite[Lemma 2.3.64]{Ab}.

To prove the statement concerning $\ell(\gamma)$ it is enough to show that, given a sequence $(z_n,w_n)\to (\zeta,\zeta)$ there exist a subsequence (still denoted $(z_n,w_n)$) such that the geodesic segment $\gamma_n\colon [0,a_n]\rightarrow \Omega$ joining $z_{n}$ to $w_{n}$ has euclidean length converging to $0$ as $n\to \infty$. 

 Let $\varphi_n$ be a complex geodesic passing through $z_n$ and $w_n$. 
 Up to composing $\varphi_n$ with an automorphism of $\D$, we may assume that $\varphi_n\in\mathscr F$. 
 Up to passing to a subsequence we have that $\varphi_n\to\varphi_\infty$  uniformly on compact subsets of $\D$. It follows from   (\ref{huang}) that actually $\varphi_n\to\varphi_\infty$ and $\varphi'_n\to\varphi'_\infty$ uniformly on $\overline\D$.
 Since strongly convex domains have simple boundary \cite[Corollary 2.1.14]{Ab}, either $\varphi_\infty\colon \D\rightarrow \Omega$ or $\varphi_\infty\equiv \zeta$.
 Let $\tau_n,\sigma_n\in\D$ be defined by  $\tau_n:=\varphi_n^{-1}(z_n)$ and that $\sigma_n:=\varphi_n^{-1}(w_n)$.

Assume that $\varphi_\infty\colon \D\rightarrow \Omega$. After taking a subsequence of $\varphi_n$ if necessary, we may assume that $\tau_n\to\tau_\infty\in \overline \D$, that $\sigma_n\to \sigma_\infty\in \overline \D$. Clearly $\tau_\infty$ and $\sigma_\infty$ belong to $\partial \D$, and $\varphi_\infty(\tau_\infty)=\varphi_\infty(\sigma_\infty)=\zeta.$
 By the continuity of the Kobayashi distance, it follows that $\varphi_\infty$ is a complex geodesic, and 
since the extension of a complex geodesic to $\overline{\D}$ is  injective,
we  obtain that $\tau_\infty=\sigma_\infty$. 
Let $\eta_n:[0,a_n]\rightarrow \D$ be the geodesic segment connecting $\tau_n$ and $\sigma_n$. Notice that since $\tau_n$ and $\sigma_n$ converge to the same point we must have $\ell(\eta_n)\to 0$. By the uniqueness of real geodesics we have $\gamma_n=\varphi_n\circ \eta_n$ and therefore
$$
\ell(\gamma_n)=\int_0^{a_n}\|\gamma_n'(t)\|\,dt\le \sup_{t\in [0,a_n]}\|\varphi_n'\left(\eta_n(t)\right)\| \ell(\eta_n).
$$
Since the value of $\sup_{\tau\in\D}\Vert\varphi_n'(\tau)\Vert$ is uniformly bounded it follows that $\ell(\gamma_n)\to 0$.

Assume now that  $\varphi_{\infty}\equiv \zeta$. 
Let $\eta_n:[0,a_n]\rightarrow \D$ be the geodesic segment connecting $\tau_n$ and $\sigma_n$. As before
$$
\ell(\gamma_n)\le \sup_{t\in [0,a_n]}\|\varphi_n'\left(\eta_n(t)\right)\| \ell(\eta_n).
$$
Since the euclidan length of geodesics lines in the disk is bounded from above, and $\varphi'_n$ converges uniformly to 0 on $\overline \D$, we have the result.

%
%
%
\end{proof}


We are now ready to prove Proposition \ref{changeofcoordinates}. We denote by $\kappa$ the Kobayashi-Royden metric.
\begin{proof}[Proof of Proposition \ref{changeofcoordinates}] Consider the change of coordinates and the constants $\rho,R>0$ given by Lemma \ref{coordinates}. Given $0<\delta<\rho$ we define the bounded sets $D:=\Omega\cup \B^q$, $D_1:=\B^q$,  $D_0:=\B^q\cap B(e_1,\delta)$.  Then by  \cite[Theorem 2.1]{FR}
we conclude, up to taking a smaller $\delta$ so that $d(\,\cdot\,,\partial D)=d(\,\cdot\,,\partial \B^q)$ on $D_0$, that there exists a constant $c>0$ so that for all  $z\in \B^q\cap B(e_1,\delta)$ and $v\in T_z\mathbb C^q$
$$
\kappa_\Omega(z;v)\ge \kappa_D(z;v)\ge (1-c\,d(z,\partial \B^q))\kappa_{\B^q}(z;v)\geq 
\kappa_{\B^q}(z;v)-c\Vert v\Vert,
$$
where the estimate $d(z,\partial \B^q)\kappa_{\B^q}(z;v)\le\Vert v\Vert$ follows from the definition of Kobayashi-Royden metric.

%

For every $\varepsilon>0$, by the previous lemma we can choose $0<\delta_1<\delta$ so that for every $z,w\in\Omega\cap B(e_1,\delta_1)$ the geodesic segment $\gamma$ for the Kobayashi distance connecting $z$ and $w$ is contained in $\Omega\cap B(e_1,\delta)$ and has  euclidean length $\ell(\gamma)\le\varepsilon/c$ (notice that $\Omega$  is not necessarily strongly convex, but the lemma still holds after a change of coordinate in ${\rm Aut}(\C^q)$). It follows that
$$
k_\Omega(z,w)\ge k_{\B^q}(z,w)-c\,\ell(\gamma)\ge k_{\B^q}(z,w)-\varepsilon.
$$

Given $R,\varepsilon>0$ as above we define choose $R_\varepsilon>0$ as in Lemma \ref{eqregion}. If $k_E$ denotes the Kobayashi distance of the horosphere $E_{\B^q}(0,e_1,R)$ we conclude that 
$$
k_\Omega(z,w)\le k_E(z,w)\le k_{\B^q}(z,w)+\varepsilon,\qquad\forall z,w\in E_{\B^q}(0,e_1,R_\varepsilon).
$$ 
By taking $R_\varepsilon$  smaller if necessary we may further assume that $E_{\B^q}(0,e_1,R_\varepsilon)\subset \Omega\cap B(e_1,\delta_1)$, concluding the proof of the Proposition. 
\end{proof}

\subsection{Proof of T{}heorem \ref{backward}}

Consider the change of coordinates given by Proposition \ref{changeofcoordinates}. We remark that in the new coordinates the domain $\Omega$ is strongly pseudoconvex but not necessarily strongly convex. On the other hand all the properties of strongly convex domains we will use in this last section are invariant under automorphisms  of $\C^q$. 

Given a decreasing sequence $0<\varepsilon_n<1/2$ converging to $0$, by Lemma \ref{changeofcoordinates} we may find another decreasing sequence $R_n>0$ so that, for all $n\geq 0$, we have $E_{\B^q}(0,e_1,R_n)\subset \Omega$ and
\begin{equation} 
\label{3eqregion}
k_{\B^q}(z,w)-\varepsilon_n\le k_\Omega(z,w)\le k_{\B^q}(z,w)+\varepsilon_n,\qquad \forall z,w\in E_{\B^q}(0,e_1,R_n).
\end{equation}

The point $e_1$ is a boundary repelling fixed point for (the conjugate of) the map $f$ with dilation $\lambda:=\lambda_\zeta$. If the map $f$ has no interior fixed point, then its Denjoy-Wolff point $\xi$ does not coincide with $e_1$, and therefore, by taking a smaller  $R_0$ if necessary, we may assume that $\xi\not\in\overline {E_{\B^q}(0,e_1,R_0)}$. On the other hand, if $f$ admits interior fixed points, there exists a limit manifold $\mathcal M$ which is a holomorphic retract of $\Omega$. Thanks to \cite[Proposition 3.4]{AbBr}, we have that $e_1\not\in \overline{ \mathcal M}$.
Hence, by taking a smaller $R_0$ if necessary, we may assume that $\overline{E_{\B^q}(0,e_1,R_0)} \cap \mathcal M =\emptyset$. We conclude that 

\begin{lemma}
\label{birba2}
Every orbit starting in  $E_{\B^q}(0,e_1,R_0)$ eventually leaves the same set. 
\end{lemma}

Set $\widetilde R:=\lambda e$,  $r_n:=R_n/\widetilde R$ and choose $z_n\in E_{\B^q}(0,e_1,r_n)$. Since $r_n<R_n<R_0$, we have $z_n\in\Omega$, and it is easy to see using  \eqref{differentpole}  that we have the two following chains of strict inclusion:
\begin{equation}\label{catenauno}
 E_{\B^q}(0,e_1,R_n)\supset E_{\B^q}(0,e_1,r_n)\supset E_{\B^q}(z_n,e_1,1).
\end{equation} 
\begin{equation}\label{catenadue}
 E_{\B^q}(0,e_1,R_n) \supset E_{\B^q}(z_n,e_1,\tilde R)\supset E_{\B^q}(z_n,e_1,1).
\end{equation}

Since  $z_n\in\Omega$ there exists a unique complex geodesic $\varphi_n$ of the domain $\Omega$ so that $\varphi_n(0)=z_n$ and $\varphi_n(1)=e_1$.
As a consequence of \cite[Lemma 3.5]{GS},
 the restriction $\alpha_n\colon[0,1)\to \C^q$ of $\varphi_n$ extends $C^1$-smoothly to the  closed interval $[0,1]$ and 
 $\alpha'(1)\not\in T_{e_1}\partial\Omega,$
 hence as the real number $t$ increases to $1$, the point $\varphi_n(t)$ converges to $e_1$ non-tangentially.

 It follows that every real $t$ sufficiently close to $1$ we have $\varphi_n(t)\in E_{\B^q}(0,e_1,r_n)$. 
 After rescaling the complex geodesic $\varphi_n$ via an appropriate automorphism of the unit disk and eventually replacing $z_n$ with $\varphi_n(0)$, we may therefore assume that $\varphi_n([0,1))\subset E_{\B^q}(0,e_1,r_n)$.

We now show that $\varphi_n\left([t_0,1)\right)\subset  E_{\B^q}(z_n,e_1,1)$, where $t_0:=\frac{e-1}{e+1}$. 
Indeed, if $0<t<t'<1$ by \eqref{3eqregion} we obtain that
\begin{align*}
 \log h^{\B^q}_{e_1,z_n}(\varphi_n(t))&=\lim_{\B^q\ni w\to e_1}[k_{\B^q}(\varphi_n(t),w)-k_{\B^q}(z_n,w)]\\
&=\lim_{t'\to 1}[k_{\B^q}(\varphi_n(t),\varphi_n(t'))-k_{\B^q}(z_n,\varphi_n(t'))]\\
&\le \lim_{t'\to 1}[k_\Omega(\varphi_n(t),\varphi_n(t'))-k_\Omega(z_n,\varphi_n(t'))]+2\varepsilon_n\\
&\le -k_{\Omega}(z_n,\varphi_n(t)) + 2\varepsilon_n\\
&<-\log\frac{1+t}{1-t}+1,
\end{align*}
where we used the fact that the three points $\varphi_n(t),\varphi_n(t')$ and $z_n$ lie on the same real geodesic of $\Omega$. Notice that we could use \eqref{3eqregion} since $\varphi_n(t),\varphi_n(t')$ and $z_n$ all belong to $E_{\B^q}(0,e_1,R_n)$.

Choose an increasing sequence $t_0<t_k<1$, converging to $1$. 
Since every orbit starting from a point in $E_{\B^q}(0,e_1,R_0)$ eventually leaves the same set, it follows that for every $n\ge 0$ we may define $m_{n,k}$ as the first positive integer so that $f^{m_{n,k}}\circ\varphi_n(t_k)\not\in  E_{\B^q}(z_n,e_1,1)$.
\begin{lemma}\label{jojorabbit}
There exists $n\in\mathbb N$ so that $f^{m_{n,k}}\circ\varphi_n(t_k)$ has a convergent subsequence in $\Omega$.
\end{lemma}
\begin{proof}
Since $\varphi_n(t_k)\in E_{\B^q}(z_n,e_1,1)$, we must have $m_{n,k}\ge 1$, and thus we may write
\begin{align*}
x_{n,k}&:=f^{m_{n,k}-1}\circ\varphi_n(t_k)\in E_{\B^q}(z_n,e_1,1)\\
y_{n,k}&:=f^{m_{n,k}}\circ\varphi_n(t_k)\not\in E_{\B^q}(z_n,e_1,1),
\end{align*}
By Lemma \ref{AbateRaissy}, for every $n\in\mathbb N$ we have
\begin{equation}\label{1917}
\limsup_{k\to\infty}k_\Omega(x_{n,k},y_{n,k})\le\lim_{k\to\infty}k_\Omega(\varphi_n(t_k),f\circ \varphi_n(t_k))=\log\lambda.
\end{equation}
Suppose now that the statement of  Lemma \ref{jojorabbit} is false. Then for every fixed  $n$ we would have that the sequences $(x_{n,k})$ and $(y_{n,k})$ both converge to $e_1$. Indeed, for the sequence $(x_{n,k})$, which is contained in  $E_{\B^q}(z_n,e_1,1)$, this follows from 
\begin{equation}
\label{birba}
\overline{E_{\B^q}(z_n,e_1,1)}\setminus\{e_1\}\subset E_{\B^q}(0,e_1,R_0)\subset \Omega.
\end{equation}
Since  $k_\Omega(x_{n,k}, y_{n,k})$ is bounded,  it  follows that  $(y_{n,k})$ converges to $e_1$ too. This is a direct consequence of \cite[Corollary 2.3.55]{Ab} and the fact that $\Omega$ is biholomorphic to a bounded strongly convex domain via an automorphism of $\C^q$.

By Definition \ref{defdil} and by \eqref{1917} we have that
\begin{align*}
\log\lambda&\le \liminf_{k\to\infty}k_\Omega(z_n,x_{n,k})-k_\Omega(z_n,y_{n,k})\\
&\le \liminf_{k\to\infty}k_\Omega(x_{n,k},y_{n,k})\\
&\le \limsup_{k\to\infty}k_\Omega(x_{n,k},y_{n,k})\\
&\le\log\lambda,
\end{align*}
proving that $\lim_{k\to\infty}k_\Omega(x_{n,k},y_{n,k})=\log\lambda$. It is now easy to show that
\begin{equation}
\label{doubleconv}
\lim_{k\to\infty}k_\Omega(z_n,x_{n,k})-k_\Omega(z_n,y_{n,k})= \lim_{k\to\infty}k_\Omega(x_{n,k},y_{n,k})=\log\lambda.
\end{equation}

Since $\log \tilde R=\log\lambda+1$, for all $n\geq 0$ we may choose a positive integer $k_n$ so that for all $k\geq k_n$,
\begin{equation}\label{killbill}
k_\Omega(x_{n,k},y_{n,k})<\log \widetilde R-2\varepsilon_n.
\end{equation}

We now show that
\begin{equation}\label{tuttinellorosfera}
z_n,x_{n,k},y_{n,k}\in E_{\B^q}(0,e_1,R_n),\qquad \forall k\ge k_n.
\end{equation}
This is clear for $z_n$ and $x_{n,k}$ by \eqref{catenauno}.
Let $\gamma_{n,k}:[0,a_{n,k}]\rightarrow \Omega$ be the geodesic segment connecting $x_{n,k}$ and $y_{n,k}$. 
Clearly for all $t\in [0,a_{n_k})$ we have  $k_{\Omega}(x_{n,k},\gamma_{n,k}(t))<k_{\Omega}(x_{n,k},y_{n,k})$.
 By \eqref{catenadue}, as long as $\gamma_{n,k}(t)\in E_{\B^q}(z_n,e_1,\widetilde R)$  we have that
\begin{align*}
\log h^{\B^q}_{e_1,z_n}(\gamma_{n,k}(t))&=\lim_{w\to e_1}[k_{\B^q}(\gamma_{n,k}(t),w)-k_{\B^q}(z_n,w)]\\
&\le \lim_{w\to e_1}[k_{\B^q}(x_{n,k},w)-k_{\B^q}(z_n,w)]+k_{\B^q}(x_{n,k},\gamma_{n,k}(t))\\
&<k_\Omega(x_{n,k},\gamma_{n,k}(t))+\varepsilon_n\\
&<\log\widetilde R-\varepsilon_n,
\end{align*}
where we used \eqref{3eqregion},\eqref{killbill} and the fact that $x_{n,k}\in E_{\B^q}(z_n,e_1,1)$. We conclude that  $\gamma_{n,k}(t)\in E_{\B^q}(z_n,e_1,\widetilde R)$ for every $t\in [0,a_{n,k}]$, and thus \eqref{tuttinellorosfera} follows.

Thus, using  \eqref{3eqregion} and  \eqref{doubleconv}, we obtain for all $k\geq k_n$,
\begin{align*}
\log\lambda-2\varepsilon_n&\le\liminf_{k\to\infty} k_{\B^q}(z_n,x_{n,k})-k_{\B^q}(z_n,y_{n,k})\\
&\le\limsup_{k\to\infty}k_{\B^q}(x_{n,k},y_{n,k})\\
&\le \log\lambda+\varepsilon_n
\end{align*}

Let $\sigma_n\in \textrm{Aut}(\B^q)$ be such that $\sigma_n(z_n)=0$ and $\sigma_n(e_1)=e_1$. Notice that for every $n$ the sequences $\sigma_n(x_{n,k}),\sigma_n(y_{n,k})\to e_1$ as $k\to\infty$. We may therefore choose a sequence $K_n\ge k_n$ so that $x_n':=\sigma_n(x_{n,K_n})\to e_1$, $y_n':=\sigma_n(y_{n,K_n})\to e_1$ and 
$$
\log\lambda-3\varepsilon_n\le  k_{\B^q}(0,x'_n)-k_{\B^q}(0,y'_n)\le k_{\B^q}(x'_n,y'_n)\le\log\lambda+2\varepsilon_n.
$$
Finally notice that $x_n'\in E_{\B^q}(0,e_1,1)$ and that $y_n'\not\in E_{\B^q}(0,e_1,1)$, which contradicts Proposition \ref{limits} since $\varepsilon_n\to 0$.
\end{proof}

By the previous lemma there exists $n\in\mathbb N$, which from now on will be fixed, such that, up to passing to a subsequence of $t_k$ if necessary, the sequence $f^{m_{n,k}}\circ\varphi_n(t_k)\to z_0\in\Omega$. By Lemma \ref{AbateRaissy} we have that for all $j\in \N$ there exists $C_j>0$ such that 
$$k_\Omega(\varphi_n(t_k),f^j\circ\varphi_n(t_k))\le C_j, \quad \forall k\geq 0.$$
Therefore, since $\varphi_n(t_k)\to e_1$, the sequence $m_{n,k}$ is divergent.

The remaining of the proof is similar to \cite{ArGu} and \cite{O}, but we add it for the convenience of the reader.
Consider the sequence $(f^{m_{n,k}-1}\circ\varphi_n(t_k))$.
Since $$k_\Omega(f^{m_{n,k}}\circ\varphi_n(t_k), f^{m_{n,k}-1}\circ\varphi_n(t_k))\leq k_\Omega (f\circ\varphi_n(t_k),\varphi_n(t_k))\to\log\lambda,$$
we can extract a subsequence $k_1(h)$ such that
$f^{m_{n,k_1(h)}-1}\circ\varphi_n(t_{k_1(h)})\to z_1\in \Omega.$
Iterating this procedure, we obtain for every $\nu\geq 1$ a subsequence $k_{\nu+1}(h)$ of $k_\nu(h)$ such that
$$
f^{m_{n,k_{\nu+1}(h)}-\nu-1}\circ\varphi_n(t_{k_{\nu+1}(h)})\to z_{\nu+1}\in \Omega,
$$
and $f(z_{\nu+1})=z_{\nu}.$
Hence $(z_\nu)$ is a backward orbit. 
 
We now show that $z_\nu\to e_1$.
Recall that $f^{m_{n,k}-\nu}\circ\varphi_n(t_k)\in E_{\B^q}(z_n,e_1,1)$, which implies that $z_\nu\in \overline{E_{\B^q}(z_n,e_1,1)}\setminus\{e_1\}$. Therefore either $z_\nu\to e_1$ or there exists a subsequence $z_{\nu_m}\to z'\in \overline{E_{\B^q}(z_n,e_1,1)}\setminus\{e_1\}\subset\Omega$. In the second case for every $i\in\mathbb N$ we have that
$$
f^i(z')=\lim_{m\to\infty}f^i(z_{\nu_m})=\lim_{m\to\infty}z_{\nu_m-i}\in \overline{E_{\B^q}(z_n,e_1,1)}\setminus\{e_1\},
$$
and thus by \eqref{birba} it follows that the orbit of the point $z'$ is contained in $E_{\B^q}(0,e_1,R_0)$, contradicting Lemma \ref{birba2}.

We are left with showing that the step of $(z_\nu)$ is $\log{\lambda}$. Let $p\in \Omega$. We have that
$$k_\Omega(z_\nu,z_{\nu-1})= k_\Omega(z_\nu,f(z_\nu)) \geq   k_\Omega(p,z_\nu)-k_\Omega(p,f(z_\nu)),   $$
and since  $z_\nu\to e_1$, it follows that  $s_1(z_\nu)\geq \log\lambda$.
Moreover,
\begin{align*}
k_{\Omega}(z_{\nu},z_{\nu-1})&=\lim_{h\to\infty}k_{\Omega}(f^{m_{n,k_{\nu}(h)}-\nu}\circ\varphi_n(t_{k_{\nu}(h)}),f^{m_{n,k_{\nu}(h)}-\nu+1}\circ\varphi_n(t_{k_{\nu}(h)})) \\&\leq \lim_{h\to\infty} k_{\Omega}(\varphi_n(t_{k_{\nu}(h)}),f\circ\varphi_n(t_{k_{\nu}(h)}))\\
&= \log\lambda.
\end{align*}

This ends the proof of Theorem \ref{backward}.

\bibliographystyle{amsplain}

\end{document}